\documentclass[a4paper,11pt,twoside]{article}

\usepackage{geometry}
\usepackage[utf8]{inputenc}
\usepackage{lmodern}
\usepackage[T1]{fontenc}
\usepackage{dsfont}

\usepackage{amsmath,amssymb,amsthm,empheq,cases}
\usepackage{hyperref}
\hypersetup{colorlinks,
            citecolor=red, 
            filecolor=black,
            linkcolor=blue,
            urlcolor=black}
\usepackage{tikz}
\usepackage{graphicx}
\usepackage{url}

\def \dis {\displaystyle}

\def \NN {\mathbb N}

\def \RR {\mathbb R}
\def \CC {\mathbb C}

\def \A {\mathcal{A}}

\def \D {\mathcal{D}}

\def \F {\mathcal{F}}

\def \K {\mathcal{K}}
\def \L {\mathcal{L}}

\def \R {\mathcal{R}}
\def \S {\mathcal{S}}
\def \T {\mathcal{T}}

\def \ecart {\noalign{\medskip}}

\theoremstyle{definition}
\newtheorem{Th}{Theorem}[section]
\newtheorem{Prop}[Th]{Proposition}
\newtheorem{Lem}[Th]{Lemma}
\newtheorem{Cor}[Th]{Corollary}
\newtheorem{Def}[Th]{Definition}
\newtheorem{Rem}[Th]{Remark}
 
\def \refs #1{Section~\ref{#1}}

\def \refT #1{Theorem~\ref{#1}}
\def \refL #1{Lemma~\ref{#1}}
\def \refC #1{Corollary~\ref{#1}}
\def \refP #1{Proposition~\ref{#1}}

\title{Bounded imaginary powers of generalized diffusion operators}

\author{Alexandre Thorel \\ \ecart
{\scriptsize Université Le Havre Normandie, Normandie Univ, LMAH UR 3821, 76600 Le Havre, France.}\\
{\scriptsize alexandre.thorel@univ-lehavre.fr}}
\date{}

\begin{document}
\maketitle

\begin{abstract}
In this paper, we investigate the boundedness of the imaginary powers of four generalized diffusion operators. This key property, which implies the maximal regularity property, allows us to solve both the linear and semilinear Cauchy problems associated with each operator. Our approach relies on semigroup theory, functional calculus, operator sum theory and $\R$-boundedness techniques to establish the boundedness of the imaginary powers of generalized diffusion operators. We then apply the Dore-Venni theorem to solve the linear problem, obtaining a unique solution with maximal regularity. Finally, we tackle the semilinear problem and prove the existence of a unique global solution.  
\end{abstract}

\section{Introduction}

The aim of this article is to show the existence and uniqueness of the global solution of four abstract semilinear evolution problems in $L^p$ framework:
\begin{equation}\label{Pb Cauchy non lineaire}
\left\{\begin{array}{l}
u'(t) - \A_i u(t) = F(t,u(t)), \quad t\in [0,T] \medskip \\ 
u(0) = \varphi,
\end{array}\right.
\end{equation}
where $u$ is a density, $T > 0$, $F$ is Lipschitz continuous, $\varphi$ is in a suitable interpolation space and $\A_i$, $i=1,2,3,4$ are defined in \cite{LMT} and recall below for the reader convenience. To this end, we study the boundedness of the imaginary powers of operators $\A_i$. Note that, to show this property on operator $\A_i$, we explicit the link between BIP operators and $\R$-boundedness theory. 

Indeed, our main application, see \refs{Sect Application}, is to study a generalized diffusion equation of Cahn-Hiliard type set on $(0,T) \times \Omega$ where $T>0$ and $\Omega$ is a cylindrical domain such that $\Omega = (a,b) \times \omega$ where $\omega$ is a bounded regular open set of $\RR^{n-1}$, $n\geqslant2$:
$$\frac{\partial u}{\partial t} (t,x,y) + \Delta^2 u(t,x,y) - k \Delta u(t,x,y) = F(t,u(t,x,y)), \quad t \in [0,T],~x \in (a,b), ~y \in \omega,$$
where $k \in \RR$. This reaction diffusion equation is supplemented by an initial condition and four kind of boundary conditions. Such a problem composed with a linear combination of the laplacian and the biharmonic operator is called generalized diffusion problem. The biharmonic term represents the long range diffusion, whereas the laplacian represents the short range diffusion.

This kind of problem arises in various concrete applications in physics, engineering and biology. For instance, in elasticity problems, we can cite \cite{elasticity 1}, \cite{elasticity 2} or \cite{IJPAM}. In electrostatic, we refer to \cite{CV 1}, \cite{electrostatic 1} or \cite{electrostatic 2} and in plates theory, we refer to \cite{hassan}, \cite{plaque 1} or \cite{plaque 2}. In complex network or more general applications, we refer to \cite{ACT}, \cite{BBT}, \cite{GCAT}. In population dynamics, we also refer to \cite{cohen-murray}, \cite{DMT}, \cite{LLMT}, \cite{LMMT}, \cite{cone 1}, \cite{cone 2}, \cite{ochoa}, \cite{okubo}, \cite{thorel 1}, \cite{thorel 2} or \cite{thorel 3} and references therein cited.

Let us recall operators $\A_i$, $i=1,2,3,4$, defined in \cite{LMT}
$$\left\{\begin{array}{ccl}
D(\A_{i}) & = & \{u \in W^{4,p}(a,b;L^p(\omega))\cap L^p(a,b;D(A^2))\text{ and } u'' \in L^p(a,b;D(A)) : \text{(BCi)}\} \\ \ecart
\left[\A_{i} u\right](x) & = & - u^{(4)}(x) - (2A - kI) u''(x) - (A^2 - k A) u(x), \quad u \in D(\A_{i}),~x\in (a,b),
\end{array}\right.$$
with $A$ a closed linear operator which satisfies some elliptic assumptions described in \refs{Sect Assumptions and main results}. Here, (BCi), $i=1,2,3,4$, represents the following boundary conditions:
\begin{equation}\tag{BC1}\label{condbord1}
\left\{\begin{array}{rclrcl}
u(a) & = & 0,   & u(b) & = & 0, \\
u''(a) & = & 0, & u''(b) & = & 0,
\end{array}\right.
\end{equation}
\begin{equation}\tag{BC2}\label{condbord2}
\left\{\begin{array}{rclrcl}
u'(a) & = & 0, & u'(b) & = & 0, \\
u''(a) + A u(a) & = & 0, & u''(b) + A u(b) & = & 0,
\end{array}\right.
\end{equation}
\begin{equation}\tag{BC3}\label{condbord3}
\left\{\begin{array}{rclrcl}
u(a)  & = & 0,  & u(b)  & = & 0, \\
u'(a) & = & 0,  & u'(b) & = & 0,
\end{array}\right.
\end{equation}
or
\begin{equation}\tag{BC4}\label{condbord4}
\left\{\begin{array}{rclrcl}
u'(a)  & = & 0, & u'(b)  & = & 0, \\
u''(a) & = & 0, & u''(b) & = & 0.
\end{array}\right.
\end{equation}

We follow the work done in \cite{LMT} where the authors have studied the associated linear problem to \eqref{Pb Cauchy non lineaire} and proved that there exists $c>0$ such that $-\A_i + cI$, $i=1,2,3,4$, are sectorial operators. Nevertheless, this property of sectoriality is not sufficient to solve the semilinear problem \eqref{Pb Cauchy non lineaire}. To this end, we need the maximal regularity property. This is why in the present paper, we study the boundedness of the imaginary powers of $-\A_i + cI$ which implies that the linear evolution problem has a unique classical solution, that is a solution with the maximal regularity. This last property leads us to prove that there exists a unique global solution to problem \eqref{Pb Cauchy non lineaire}.

This paper is organized as follows. In \refs{Sect Recall}, we recall some definitions and usefull lemmas. In \refs{Sect Assumptions and main results}, we state our assumptions and main results. Then, in \refs{Sect Resultats Techniques}, we prove all technical results we will use in the proofs of our main results. \refs{Sect Proofs of main results} is devoted to the proofs of our main results. Finally, in \refs{Sect Application}, we give an application.

\section{Recall and prerequisites}\label{Sect Recall}

\begin{Def}\label{Def Banach}
A Banach space $X$ is a UMD space if and only if for all $p\in(1,+\infty)$, the Hilbert transform is bounded from $L^p(\RR,X)$ into itself (see \cite{bourgain} and \cite{burkholder}).
\end{Def}
\begin{Def}\label{Def op Sect}
Let $\omega_{T_1} \in(0,\pi)$. Sect($\omega_{T_1}$) denotes the space of closed linear operators $T_1$ which satisfies
$$
\begin{array}{l}
i)\quad \sigma(T_1)\subset \overline{S_{\omega_{T_1}}},\\ \ecart
ii)\quad\forall~\omega_{T_1}'\in (\omega_{T_1},\pi),\quad \sup\left\{\|\lambda(T_1 - \lambda\,I)^{-1}\|_{\L(X)} : ~\lambda\in\CC\setminus\overline{S_{\omega_{T_1}'}}\right\}<+\infty,
\end{array}
$$
where
\begin{equation}\label{defsector}
S_{\omega_{T_1}} \;:=\;\left\{\begin{array}{lll}
\left\{ z \in \CC : z \neq 0 ~~\text{and}~~ |\arg(z)| < \omega_{T_1} \right\} & \text{if} & \omega_{T_1} \in (0, \pi], \\ \ecart
\,(0,+\infty) & \text{if} & \omega_{T_1} = 0,
\end{array}\right.
\end{equation} 
see \cite{haase}, p. 19. Such an operator $T_1$ is called sectorial operator of angle $\omega_{T_1}$.
\end{Def}
\begin{Rem}
From \cite{komatsu}, p. 342, we know that any injective sectorial operator~$T_1$ admits imaginary powers $T_1^{is}$, $s\in\RR$, but, in general, $T_1^{is}$ is not bounded.
\end{Rem} 
\begin{Def}
Let $\theta \in [0, \pi)$. We denote by BIP$(X,\theta)$, the class of sectorial injective operators $T_2$ such that
\begin{itemize}
\item[] $i) \quad ~~\overline{D(T_2)} = \overline{R(T_2)} = X,$

\item[] $ii) \quad ~\forall~ s \in \RR, \quad T_2^{is} \in \L(X),$

\item[] $iii) \quad \exists~ C \geq 1 ,~ \forall~ s \in \RR, \quad ||T_2^{is}||_{\L(X)} \leq C e^{|s|\theta}$,
\end{itemize}
see~\cite{pruss-sohr}, p. 430.
\end{Def}

\begin{Lem}\label{Lem trace}
Let $T_3$ be a linear operator satisfying 
$$(0,+\infty) \subset \rho(T_3)\quad \text{and}\quad \exists~C>0:\forall~t>0, \quad \|t(T_3-tI)^{{\mbox{-}}1}\|_{\L(X)} \leqslant C.$$
Let $u$ be such that
$$u \in W^{n,p}(a,b;X) \cap L^p(a,b;D(T_3^m)),$$ 
where $a,b \in \RR$ with $a<b$, $n,m \in \NN\setminus\{0\}$ and $p \in (1,+\infty)$. Then for any $j \in \NN$ satisfying the Poulsen condition: 
$$0<\frac{1}{p}+j < n,$$ 
and $s \in \{a,b\}$, we have
$$
u^{(j)}(s) \in (D(T_3^m),X)_{\frac{j}{n}+\frac{1}{np},p}.
$$
\end{Lem}
This result is proved in \cite{grisvard}, Teorema 2', p. 678.

\begin{Lem}\label{Lem triebel}
Let $\psi \in X$ and $T_3$ be a generator of a bounded analytic semigroup in $X$ with $0 \in \rho(T_3)$. Then, for any $m \in \NN\setminus\{0\}$ and $p\in [1,+\infty]$, the next properties are equivalent:
\begin{enumerate}

\item $x\mapsto T_3^m e^{(x-a)T_3} \psi \in L^p(a,+\infty;X)$

\item $\psi \in \left(D(T_3),X\right)_{m-1+\frac{1}{p},p}$

\item $x\mapsto e^{(x-a)T_3} \psi \in W^{m,p}(a,b;X)$

\item $x\mapsto T_3^m e^{(x-a)T_3} \psi \in L^p(a,b;X)$.
\end{enumerate}
\end{Lem}
The equivalence between 1 and 2 is proved in \cite{triebel}, Theorem, p. 96. The others are proved in \cite{thorel 1}, Lemma 3.2, p. 638-639.

\section{Assumptions and main results}\label{Sect Assumptions and main results}

Let $A$ be a closed linear operator and assume
\begin{center}
\begin{tabular}{l}
$(H_1)\quad X$ is a UMD space, \\ \ecart
$(H_2)\quad 0 \in \rho(A)$, \\ \ecart
$(H_3)\quad -A \in$ BIP$(X,\theta_A),~$ for $\theta_A \in [0, \pi/2)$,\\ \ecart
$(H_4)\quad [k,+\infty) \in \rho(A)$.
\end{tabular}
\end{center}

\begin{Th}\label{Th BIP}
Let $\theta_A \in [0, \pi/2)$. Assume that $(H_1)$, $(H_2)$, $(H_3)$ and $(H_4)$ hold. Then
\begin{enumerate}
\item for $i=1,2$, for all $\varepsilon >0$, operator $\dis -\A_i+\frac{k^2}{4}I \in$ BIP\,$(X,2\theta_A + \varepsilon)$,

\item  for $i=3,4$, for all $\varepsilon > 0$, operator $\dis -\A_i+\frac{k^2}{4}I + r'I \in$ BIP\,$(X,2\theta_A + \varepsilon)$.
\end{enumerate}
\end{Th}

\begin{Th}\label{Th Pb Cauchy}
Let $T>0$, $f \in L^p(0,T;X)$, $p \in (1,+\infty)$ and $\theta_A \in [0,\pi/4)$. Assume that $(H_1)$, $(H_2)$, $(H_3)$ and $(H_4)$ hold. Then, for $i=1,2,3,4$, there exists a unique classical solution of problem 
\begin{equation}\label{P}
\left\{\begin{array}{l}
\dis u'(t) - \A_i u(t) = f(t), \quad t \in(0,T] \\ \ecart
u(0) = u_0,
\end{array}\right.
\end{equation}
if and only if
$$u_0 \in (D(\A_i),X)_{\frac{1}{p},p}.$$
\end{Th}

\begin{Cor}\label{Cor sol Pb non linéaire}
 Assume that $(H_1)$, $(H_2)$, $(H_3)$ and $(H_4)$ hold with $\theta_A \in [0,\pi/4)$. Let $T>0$ and $F:[0,T]\times (D(\A_i),X)_{\frac{1}{p},p} \longrightarrow X$, $p \in (1,+\infty)$, such that 
\begin{itemize}
\item $F(.,u)$ is measurable for each $u \in (D(\A_i),X)_{\frac{1}{p},p}$,

\item $F(t,.)$ is continuous for almost all $t \in [0,T]$.
\end{itemize} 
Moreover, assume that for all $r>0$, there exists $\phi_r \in L^p(0,T;\mathbb{R}_+)$ such that for almost all $t \in (0,T)$, $u$, $\overline{u} \in (D(\A_i),X)_{\frac{1}{p},p}$, with $\|u\|_{L^p(0,T;X)}$, $\|\overline{u}\|_{L^p(0,T;X)} \leqslant r$, we have
$$\left\|F(t,u) - F(t,\overline{u})\right\|_X \leqslant \phi_r (t) \|u - \overline{u}\|_{L^p(0,T;X)}.$$
Then, for $i=1,2,3,4$, there exists a unique global classical solution of problem 
\begin{equation}\label{P non linéaire}
\left\{\begin{array}{l}
\dis u'(t) - \A_i u(t) = F(t,u(t)), \quad t \in(0,T] \\ \ecart
u(0) = u_0,
\end{array}\right.
\end{equation}
if and only if
$$u_0 \in (D(\A_i),X)_{\frac{1}{p},p}.$$
\end{Cor}

\begin{proof}
Due to \refT{Th Pb Cauchy}, operators $-\A_i$, $i=1,2,3,4$, have the property of maximal $L^p$-regularity, see \cite{pruss 2}, Definition 1.1, p. 2. Then from \cite{pruss 2}, Theorem 3.1, p. 10, there exists $T_0 \in (0,T]$ such that $u$ is the unique solution of problem \eqref{P non linéaire} with
$$u \in W^{1,p} (0,T_0;X) \cap L^p(0,T_0; D(\A_i)),$$
if and only if
$$u_0 \in (D(\A_i),X)_{\frac{1}{p},p}.$$
Moreover, since $u \in W^{1,p} (0,T_0;X) \subset C (0,T_0;X)$, due to the Heine-Cantor theorem, $u$ is uniformly continuous. Therefore, from \cite{pruss 2}, Proposition 3.3, p. 11, $u$ exists globally on $[0,T]$.
\end{proof}

\begin{Rem}
The map $u_0 \longmapsto u(t)$ defines a local semiflow on the natural phase space $(D(\A_i),X)_{\frac{1}{p},p}$.
\end{Rem}

\section{Technical results}\label{Sect Resultats Techniques}

The following result can be find in \cite{Dore-Labbas} Proposition 4.9, p. 1879 or in \cite{dore-yakubov}, Lemma 2.3, p. 99.
\begin{Prop}\label{Prop DoreLabbas}
Let $z_1,z_2 \in \mathbb{C}\setminus\{0\}$. We have
$$|z_1+z_2| \geqslant \left(|z_1|+|z_2|\right)\left|\cos\left(\frac{\arg(z_1)-\arg(z_2)}{2}\right)\right|.$$
\end{Prop}
The following Lemma is due to \cite{labbas-terreni}, Lemma 6.1, p. 564, for the reader convenience, we recall the proof. 
\begin{Lem}\label{Lem Lab-Ter}
Let $\theta_0 \in \,(0,\pi)$. For all $z \in \overline{S_{\theta_0}}$ and all $c \in \RR_+\setminus\{0\}$, there exists $C > 0$, independant of $z$ and $c$, such that
$$\left|z + c \right| \geqslant C\, |z| \quad \text{and} \quad \left|z + c \right| \geqslant C \,|c|.$$
For all $z \in \partial\overline{S_{\theta_0}}$ and all $c \in \RR_-\setminus\{0\}$, there exists $C > 0$, independant of $z$ and $c$, such that
$$\left|z + c \right| \geqslant C\, |z| \quad \text{and} \quad \left|z + c \right| \geqslant C \,|c|.$$
\end{Lem}
\begin{proof}
When $z=0$, the result is obvious. Let us consider that $z\neq 0$. From \refP{Prop DoreLabbas}, we have
$$|z+c| \geqslant \left(|z| + |c|\right) \left|\cos\left(\frac{\arg(z)-\arg(c)}{2}\right)\right|.$$
Thus, we have to consider the two following cases:
\begin{itemize}
\item Case $c> 0$. 

Since $\arg(c) = 0$ and $|\arg(z)| \leqslant \theta_0 < \pi$, we have
$$\cos\left(\frac{\arg(z) - \arg(c)}{2}\right) = \cos\left(\frac{\arg(z)}{2}\right) \geqslant \cos\left(\frac{\theta_0}{2}\right) > 0.$$
Then $\dis C = \cos\left(\frac{\theta_0}{2}\right)$ and the result follows.

\item Case $c < 0$. 

Since $\arg(c) = \pi$ and $0 < |\arg(z)| = \theta_0 < \pi$, we have
$$\cos\left(\frac{\arg(z) - \arg(c)}{2}\right) = \cos\left(\frac{\pi}{2} \pm \frac{\theta_0}{2}\right) = \pm\sin\left(\frac{\theta_0}{2}\right) \neq 0.$$
Then $\dis C = \left|\sin\left(\frac{\theta_0}{2}\right)\right|$ and the result follows. 
\end{itemize}
\end{proof}

\begin{Lem}\label{Lem F(E)}
Let $-T \in$ Sect$\,(\omega_T)$, $\omega_T \in [0,\pi/2)$ with $0\in \rho(T)$. For all $x\in \RR$, we set
$$E_T (x) = e^{|x|T}.$$
Then
$$\F\left(E_T\right) (\xi) = -2T\left(T^2 + 4\pi^2\xi^2 I\right)^{-1}.$$ 
\end{Lem}

\begin{proof}
Using functional calculus, we have
$$\begin{array}{lll}
\dis\F\left(E_{T}\right)(\xi) & = & \dis \int_{-\infty}^{+\infty}  e^{|x| T} e^{-2i\pi\xi x} ~dx \\ \ecart

& = & \dis \int_{-\infty}^{+\infty}  \frac{1}{2i\pi}\int_\Gamma e^{-|x|z} (-T - z I)^{-1}~dz~ e^{-2i\pi\xi x} ~dx \\ \ecart

& = & \dis \frac{1}{2i\pi}\int_\Gamma \int_{-\infty}^{+\infty} e^{-|x|z} e^{-2i\pi\xi x} ~dx~ (-T - z I)^{-1}~dz \\ \ecart

& = & \dis \frac{1}{2i\pi}\int_\Gamma \F\left(e^{-|.|z}\right) (-T - z I)^{-1}~dz \\ \ecart

& = & \dis \frac{1}{2i\pi}\int_\Gamma \frac{2z}{z^2 + 4\pi^2\xi^2} (-T - z I)^{-1}~dz \\ \ecart

& = & \dis -2T \left(T^2 + 4 \pi^2\xi^2 I\right)^{-1},
\end{array}$$
where $\Gamma$ is the boundary of $S_{\eta}\setminus B(0,R_T)$, positively oriented, with $R_T > 0$ and $\eta $ fixed in $\left( \omega_T , \frac{\pi}{2} \right)$. Note that, since $-T\in$ Sect\,$(\omega_T)$ with $\omega_T \in \left[0,\frac{\pi}{2}\right)$, then from \cite{haase}, Proposition 3.1.2, p. 63, we have $T^2 \in$ Sect\,$(2\omega_T)$. Moreover, since $0\in\rho(T)$, then $0 \in \rho\left(T^2 + 4 \pi^2\xi^2 I\right)$.
\end{proof}

\begin{Prop}\label{Prop A+zI in Sect}
Let $\omega_T \in [0,\pi)$ and $z \in \CC\setminus\{0\}$ fixed such that $|\arg(z)| + \omega_T < \pi$. Then
$$T \in \text{Sect}(\omega_T) \Longrightarrow T + z I \in \text{Sect}(\max(\omega_T,|\arg(z)|)).$$
\end{Prop}

\begin{proof} 
If $\mu \in \overline{S_{\omega_T}}$, then since $\omega_T + |\arg(z)| < \pi$, from Proposition 13, p. 8 in \cite{DLMMT}, we have
$$\begin{array}{lll}
|\arg(z+\mu)| &\leqslant & \dis \left|\max\left(\arg(z), \arg(\mu)\right)\right| \leqslant \max\left(|\arg(z)|, |\arg(\mu)|\right) \\ \ecart
&\leqslant & \dis \max\left(|\arg(z)|, \omega_T \right) < \pi.
\end{array}$$
Therefore $\sigma(T+zI) \subset \overline{S_{\max(|\arg(z)|, \omega_T)}}$.

For all $\omega' \in (\max\left(|\arg(z)|, \omega_T\right),\pi)$, if $\lambda \notin \overline{S_{\omega'}}$, then $\lambda \neq 0$ and 
$$\max(|\arg(z)|,\omega_T) < \omega' < |\arg(\lambda)|.$$ 
Moreover, we have
$$\lambda \in \rho(T+zI) \quad \text{and} \quad \lambda - z \in \rho(T).$$
Thus, there exists $M > 0$ such that 
$$\left\|(\lambda-z)(T + zI - \lambda I)^{-1}\right\|_{\L(X)} = \left\|(\lambda-z)(T - (\lambda-z)I)^{-1}\right\|_{\L(X)} \leqslant M,$$
Moreover, from \refP{Prop DoreLabbas}, we have
$$|\lambda - z| \geqslant \left(|\lambda|+|z|\right)\left|\cos\left(\frac{\arg(\lambda) - \arg(-z))}{2}\right)\right| > |\lambda| \left|\cos\left(\frac{\arg(\lambda) - \arg(-z)}{2}\right)\right|,$$ 
and
$$\begin{array}{lll}
\dis\left|\cos\left(\frac{\arg(\lambda) - \arg(-z)}{2}\right)\right| &=& \dis \left|\cos\left(\frac{\arg(\lambda) - \arg(z) \pm \pi}{2}\right)\right| \medskip \\ 
&=& \dis \left|\sin\left(\frac{\arg(\lambda) - \arg(z)}{2}\right)\right| \medskip \\
&=& \dis \sin\left(\frac{|\arg(\lambda) - \arg(z)|}{2}\right).
\end{array}$$
Furthermore, since $|\arg(\lambda)| > \omega' > |\arg(z)|$ and $|\arg(z)| < \pi - \omega_T$, we obtain that
$$|\arg(\lambda) - \arg(z)| \geqslant \left||\arg(\lambda)| - |\arg(z)|\right| \geqslant \omega' - |\arg(z)| > 0,$$
and
$$|\arg(\lambda) - \arg(z)| \leqslant |\arg(\lambda)| + |\arg(z)| < \pi + \omega' < 2 \pi.$$
Then
$$0 < \frac{\omega' - |\arg(z)|}{2} < \frac{|\arg(\lambda) - \arg(z)|}{2} < \frac{\pi}{2} +\frac{\omega'}{2} < \pi,$$
and
$$\sin\left(\frac{|\arg(\lambda) - \arg(z)|}{2}\right) > \min\left(\sin\left(\frac{\omega'-|\arg(z)|}{2}\right), \sin\left(\frac{\pi}{2} + \frac{\omega'}{2}\right)\right) > 0.$$
Thus, since $\lambda \neq 0$, we have 
$$|\lambda - z| > |\lambda| \min\left(\sin\left(\frac{\omega'-|\arg(z)|}{2}\right), \sin\left(\frac{\pi}{2} + \frac{\omega'}{2}\right)\right) > 0.$$
Finally, setting 
$$C = \frac{M}{\min\left(\sin\left(\frac{\omega'-|\arg(z)|}{2}\right), \sin\left(\frac{\pi}{2} + \frac{\omega'}{2}\right)\right)} > 0,$$
we obtain
$$\left\|(A+zI - \lambda I)^{-1}\right\|_{\L(X)} \leqslant \frac{M}{|\lambda - z|} < \frac{C}{|\lambda|}.$$
Therefore $A+zI \in \text{Sect}(\max(|\arg(z)|, \omega_T))$.
\end{proof}

\begin{Cor}\label{Cor A+zI in Sect}
Let $\omega_T \in (0,\pi)$ and $z \in \CC\setminus\{0\}$ fixed such that $|\arg(z)| + \omega_T < \pi$. Then
$$T \in \text{Sect}(\omega_T) \Longrightarrow T + z I \in \text{Sect}(\max(\omega_T,\pi-\omega_T)).$$
Moreover for all $\lambda \in \CC \setminus \overline{S_{\max(\omega_T, \pi - \omega_T)}}$, we have
$$\left\|\left(T + zI - \lambda I \right)^{-1} \right\|_{\L(X)} \leqslant \frac{C}{|\lambda|},$$
where $C > 0$ does not depends on $z$ and $\lambda$.
\end{Cor}

\begin{proof}
The proof is similar to the one of \refP{Prop A+zI in Sect}, thus we only point out the differences. From the proof of \refP{Prop A+zI in Sect}, we have
$$\sigma(T+zI) \subset \overline{S_{\max(\omega_T, |\arg(z)|)}} \subset \overline{S_{\max(\omega_T, \pi - \omega_T)}}.$$
For all $\omega' \in (\max\left(\omega_T, \pi - \omega_T\right),\pi)$, if $\lambda \notin \overline{S_{\omega'}}$, then $\lambda \neq 0$ and 
$$\max(\omega_T, |\arg(z)|) \leqslant \max(\omega_T, \pi - \omega_T) < \omega' < |\arg(\lambda)|.$$ 
Thus $\lambda - z \in \rho(T)$ and there exists $M > 0$, independent of $z$ and $\lambda$, such that 
$$\left\|(\lambda-z)(T - (\lambda-z)I)^{-1}\right\|_{\L(X)} \leqslant M.$$
Moreover, from \refP{Prop DoreLabbas}, we have
$$|\lambda - z| > |\lambda| \left|\cos\left(\frac{\arg(\lambda) - \arg(-z)}{2}\right)\right| = |\lambda| \sin\left(\frac{|\arg(\lambda) - \arg(z)|}{2}\right),$$ 
where
$$|\arg(\lambda) - \arg(z)| \geqslant \left||\arg(\lambda)| - |\arg(z)|\right| \geqslant \omega' - |\arg(z)| > \omega' - (\pi - \omega_T) > 0,$$
and
$$|\arg(\lambda) - \arg(z)| \leqslant |\arg(\lambda)| + |\arg(z)| < \pi + \pi - \omega_T < 2 \pi - \omega_T.$$
Then
$$0 < \frac{\omega'}{2} - \frac{\pi - \omega_T}{2} < \frac{|\arg(\lambda) - \arg(z)|}{2} < \pi -\frac{\omega_T}{2} < \pi,$$
and
$$\sin\left(\frac{|\arg(\lambda) - \arg(z)|}{2}\right) > \min\left(\sin\left(\frac{\omega'}{2} - \frac{\pi - \omega_T}{2}\right), \sin\left(\pi - \frac{\omega_T}{2}\right)\right) > 0.$$
Finally, setting 
$$C = \frac{M}{\min\left(\sin\left(\frac{\omega'}{2} - \frac{\pi - \omega_T}{2}\right), \sin\left(\pi - \frac{\omega_T}{2}\right)\right)} > 0,$$
which only depends on $\omega_T$, we obtain
$$\left\|(A+zI - \lambda I)^{-1}\right\|_{\L(X)} \leqslant \frac{M}{|\lambda - z|} < \frac{C}{|\lambda|}.$$
Therefore $A+zI \in \text{Sect}(\max(\omega_T, \pi - \omega_T))$.
\end{proof}

\begin{Cor}\label{Cor A inv implique A+zI inv}
Let $T\in$ Sect$(\omega_T)$ with $\omega_T \in [0,\pi)$ and $z \in \CC\setminus\{0\}$ such that $\omega_T + |\arg(z)| < \pi$. Then, we have
$$0 \in \rho(T) \Longrightarrow 0 \in \rho(T + z I).$$
\end{Cor}

\begin{proof}
Since $\omega_T + |\arg(z)| < \pi$, there exists $\omega_0 > 0$ such that
$$\omega_0 = \pi - \omega_T - |\arg(z)|.$$
Moreover, $0 \in \rho(T)$, then there exists $r > 0$ such that
$$\overline{B(0,r)} \subset \rho(T),$$
where $B(0,r)$ is the open ball centered at $0$ of radius $r$. We set
$$\varepsilon = \frac{r}{2} \left|\cos(\omega_T) - \cos(\omega_T + \omega_0)\right|,$$ 
it follows that $\varepsilon \in (0,r)$ and
$$T - \varepsilon I \in \text{Sect}\left(\omega_T + \frac{\omega_0}{2}\right).$$
Finally, since
$$\omega_T + \frac{\omega_0}{2} + |\arg(z)| = \pi - \frac{\omega_0}{2} < \pi,$$
from \refP{Prop A+zI in Sect}, we have
$$T - \varepsilon I + z I \in \text{Sect}\left(\max\left(\omega_T + \frac{\omega_0}{2}, |\arg(z)|\right)\right),$$
hence
$$0 \in \rho(T - \varepsilon I + z I + \varepsilon I) = \rho(T+zI),$$ 
which gives the expected results.
\end{proof}

\begin{Cor}\label{Cor (A+z)^alp in Sect}
Let $\omega_T \in (0,\pi)$ and $z \in \CC\setminus\{0\}$ such that $|\arg(z)| + \omega_T < \pi$. Then, for all fixed $\alpha \in \left(0,\frac{\pi}{\max(\omega_T,\pi-\omega_T)}\right)$, we have
$$T \in \text{Sect}(\omega_T) \Longrightarrow \left(T + z I\right)^\alpha \in \text{Sect}(\alpha\max(\omega_T,\pi-\omega_T)).$$
Moreover, for all $\lambda \in \CC \setminus \overline{S_{\alpha\max(\omega_T, \pi-\omega_T)}}$, we deduce that
$$\left\|\left(\left(T+zI\right)^\alpha - \lambda I \right)^{-1} \right\|_{\L(X)} \leqslant \frac{C}{|\lambda|},$$
where $C > 0$ does not depends on $z$ and $\lambda$.
\end{Cor}
\begin{proof}
The result follows from \refC{Cor A+zI in Sect} and Proposition 3.1.2, p. 63 in \cite{haase}.
\end{proof}

\begin{Lem}\label{Lem T^alp e^(tT) borné}
Let $T \in$ Sect$(\omega_T)$, with $\omega_T \in (0,\pi)$, $\lambda \in \CC\setminus\{0\}$ such that $|\arg(\lambda)| + \omega_T < \pi$ and $\alpha \in \left(0, \frac{\pi}{2\max\left(\omega_T, \pi - \omega_T\right)}\right)$. We set
$$T_\lambda = -\left(T + \lambda I\right)^\alpha.$$
Then, $T_\lambda$ is well defined and generates an analytic semigroup; moreover for all $t>0$ and $k>0$, there exists $C>0$, independent of $\lambda$, such that
$$\left\|T_\lambda^k e^{tT_\lambda}\right\|_{\L(X)} \leqslant C.$$
\end{Lem}

\begin{proof}
From \refC{Cor (A+z)^alp in Sect}, $-T_\lambda$ is a sectorial operator of angle $\omega_{T_\lambda} = \alpha\max(\omega_T, \pi - \omega_T) < \pi/2$. Since $\lambda \neq 0$, then due to \refC{Cor A inv implique A+zI inv}, we have $0 \in \rho(T+\lambda I)$. From \refC{Cor A+zI in Sect} and Proposition~3.1.1, e), p. 61-62 in \cite{haase}, it follows that $0 \in \rho(T_\lambda)$. Thus there exists $r>0$ such that $\overline{B(0,r)} \subset \rho(T_\lambda)$. Then, by functional calculus, we have
$$\begin{array}{lll}
\dis T_\lambda^k e^{tT_\lambda} &=& \dis \frac{1}{2i\pi} \int_{\Gamma_{T_\lambda}} (-z)^k e^{-t z} \left(-T_\lambda - zI\right)^{-1}dz \\ \\
&=& \dis \frac{1}{2i\pi} \int_r^{+\infty} |z|^k e^{ik\omega_{T_\lambda}} e^{-t|z|e^{i\omega_{T_\lambda}}} \left(-T_\lambda - |z|e^{i\omega_T}I\right)^{-1} e^{i\omega_T}~d|z| \\ \ecart
&& \dis + \frac{1}{2i\pi} \int_{-\omega_{T_\lambda}}^{\omega_{T_\lambda}} r^ke^{ik\theta} e^{-trke^{i\theta}} \left(-T_\lambda - re^{i\theta}I\right)^{-1} ire^{i\theta}~d\theta \\ \ecart
&& \dis - \frac{1}{2i\pi} \int_r^{+\infty} |z|^k e^{-ik\omega_{T_\lambda}} e^{-t|z|e^{-i\omega_{T_\lambda}}} \left(-T_\lambda - |z|e^{-i\omega_{T_\lambda}}I\right)^{-1} e^{-i\omega_{T_\lambda}}~d|z|,
\end{array}$$
where $\Gamma_{T_\lambda}$ is a sectorial curve, positively oriented, of angle $\omega_{T_\lambda}$ which avoids $0$ at distance $r$.

From \refC{Cor (A+z)^alp in Sect}, there exists $C>0$, independent of $\lambda$, such that
$$\begin{array}{lll}
\dis \left\|T_\lambda^k e^{tT_\lambda}\right\|_{\L(X)} &\leqslant & \dis \frac{1}{\pi} \int_r^{+\infty} |z|^k e^{-t|z|\cos(\omega_{T_\lambda})} \frac{C}{|z|}~d|z| + \frac{1}{2\pi} \int_{-\omega_{T_\lambda}}^{\omega_{T_\lambda}} Cr^k e^{-trk\cos(\theta)} ~d\theta \\ \ecart

& \leqslant & \dis \frac{C}{\pi} \int_r^{+\infty} |z|^{k-1} e^{-t|z|\cos(\omega_{T_\lambda})} ~d|z| + \frac{Cr^k\omega_{T_\lambda}}{\pi}.
\end{array}$$
By successive integration by parts, there exists $C'>0$, independent of $\lambda$, such that
$$\begin{array}{lcl}
\dis \int_r^{+\infty} |z|^{k-1} e^{-t|z|\cos(\omega_{T_\lambda})} ~d|z| &=& \dis \left[ \frac{e^{-t|z|\cos(\omega_{T_\lambda})}}{-t\cos(\omega_{T_\lambda})} |z|^{k-1}\right]_r^{+\infty} - \int_r^{+\infty} \frac{e^{-t|z|\cos(\omega_{T_\lambda})}}{-t\cos(\omega_{T_\lambda})} (k-1)|z|^{k-2} ~d|z| \\ 
&\vdots& \\
& = & \dis C' + (-1)^{k-1} \int_r^{+\infty} \frac{e^{-t|z|\cos(\omega_{T_\lambda})}}{\left(-t\cos(\omega_{T_\lambda})\right)^{k-1}} (k-1)! ~d|z| \\ \ecart
& = & \dis C' + (-1)^{k-1} \left[ \frac{e^{-t|z|\cos(\omega_{T_\lambda})}}{\left(-t\cos(\omega_{T_\lambda})\right)^{k}}\right]_r^{+\infty} (k-1)! < +\infty,
\end{array}$$
which concluded the proof.
\end{proof}

\begin{Lem}\label{Lem Amann}
Let $T \in$ Sect$\,(\omega_T)$ with $\omega_T \in [0,\pi)$, $0 \in \rho(T)$ and $\overline{D(T)} = X$. For all $\lambda \in \rho(T)$, assume that
$$\left(T - \lambda I\right)^{-1} = R(\lambda) + S(\lambda),$$
and for Re$(z)<0$, set
$$R_z(\lambda):= \lambda^z R(\lambda).$$
Also assume that
$$R_z, S \in  L^1(\Gamma,\L(X)),$$
where $\Gamma:=\partial \overline{S_\nu}$, with $\nu \in (\omega_T,\pi)$ and for some $\varepsilon > 0$
$$\left\|\int_\Gamma R_z(\lambda)~d\lambda\right\|_{\L(X)} \leqslant M e^{\theta |\text{Im}(z)|}, \quad -\varepsilon<\text{Re}(z)<0.$$
Then, we have
$$T \in \text{BIP}\left(X,\max(\nu, \theta)\right).$$
\end{Lem}
This result can be found in \cite{amann 2}, Lemma 4.8.4, p. 170, but we have adapted the notation for the reader convenience.

\begin{Lem}\label{Lem B_xi BIP inversible}
Let $r' > 0$, for all $\xi \in \RR$, we set 
\begin{equation}\label{Def B_xi}
B_{\xi} := -A\left(-A + k I + 8\pi^2\xi^2 I\right) + \left(\frac{k}{2} + 4\pi^2\xi^2 \right)^2 I + r'I.
\end{equation}
Assume that $(H_1)$, $(H_2)$, $(H_3)$ and $(H_4)$ hold. Then, we have
\begin{equation*}
B_{\xi} \in \text{BIP}\,(X,2\theta_A) \quad \text{and}\quad 0\in \rho(B_{\xi}),
\end{equation*}
such that
$$\forall \,r,\xi \in \RR, \quad B_{\xi}^{ir} \in \L(X) \quad \text{and} \quad \exists\, C \geqslant 1,~\forall\, r,\xi \in \RR, \quad \left\| B_{\xi}^{ir} \right\|_{\L(X)} \leqslant C e^{2\theta_A |r|},$$
where $C$ is independent of $\xi$ and $r$.
\end{Lem}

\begin{proof}
From \cite{arendt-bu-haase}, Theorem 2.3, p. 69, $-A + kI \in$ BIP$\,(X,\theta_A)$ and due to \cite{pruss-sohr}, Theorem 3, p. 437, we have 
$$-A + k I + 8\pi^2\xi^2 I \in \text{BIP}\,(X,\theta_A),$$
such that
$$\left\{\begin{array}{l}
\forall \,r \in \RR, \quad \left(-A + kI + 8\pi^2\xi^2 I\right)^{ir} \in \L(X) \\ \ecart
\exists\, C \geqslant 1,~\forall\, r, \xi \in \RR, \quad \left\| \left(-A + kI + 8\pi^2\xi^2 I\right)^{ir} \right\|_{\L(X)} \leqslant C e^{\theta_A |r|},
\end{array}\right.$$
where $C$ is independent of $\xi$ and $r$. Moreover, since $k + 8\pi^2\xi^2 \in [k,+\infty)$ and $-A \in$ BIP\,$(X,\theta_A)$, from \cite{arendt-bu-haase}, Theorem 2.3, p. 69, $0\in \rho(-A + k I + 8\pi^2\xi^2 I)$. Hence, due to \cite{pruss-sohr}, Corollary 3, p. 444, it follows that
$$-A\left(-A + k I + 8\pi^2\xi^2 I\right) \in \text{BIP}\,(X,2\theta_A) \quad \text{and} \quad 0\in \rho(-A(-A + k I + 8\pi^2\xi^2 I)),$$
such that
$$\left\{\begin{array}{l}
\forall \,r \in \RR, \quad \left(-A\left(-A + kI + 8\pi^2\xi^2 I\right)\right)^{ir} \in \L(X) \\ \ecart
\exists\, C \geqslant 1,~\forall\, r, \xi \in \RR, \quad \left\| \left(-A\left(-A + kI + 8\pi^2\xi^2 I\right)\right)^{ir} \right\|_{\L(X)} \leqslant C e^{2\theta_A |r|},
\end{array}\right.$$
where $C$ is independent of $\xi$ and $r$. Finally, from \cite{pruss-sohr}, Theorem 3, p. 437, we have 
$$\forall \,r,\xi \in \RR, \quad B_{\xi}^{ir} \in \L(X) \quad \text{and} \quad \exists\, C \geqslant 1,~\forall\, r,\xi \in \RR, \quad \left\| B_{\xi}^{ir} \right\|_{\L(X)} \leqslant C e^{2\theta_A |r|},$$
where $C$ is independent of $\xi$ and $r$. Moreover, due to \cite{arendt-bu-haase}, Theorem 2.3, p. 69, $0\in\rho(B_\xi)$.
\end{proof}

\begin{Lem}\label{Lem dérivée B_xi^alpha}
Let $\alpha \in \CC$ and $r' > 0$ fixed, with Re$(\alpha) \leqslant 0$. For all $\xi \in \RR$, we set
$$B_{\xi} = -A\left(-A + k I + 8\pi^2\xi^2 I\right) + \left(\frac{k}{2} + 4\pi^2\xi^2 \right)^2 I + r'I.$$
Assume that $(H_1)$, $(H_2)$, $(H_3)$ and $(H_4)$ hold. Then, we have
$$\begin{array}{lll}
\dis\frac{\partial B_{\xi}^\alpha }{\partial \xi} & = & \dis 16\alpha\pi^2\xi\left(-A + \frac{k}{2}I + 4\pi^2\xi^2 I\right) B_{\xi}^{\alpha-1}, \\ \\
\dis \frac{\partial^2 B_{\xi}^\alpha }{\partial \xi^2} & = & \dis 16\alpha\pi^2\left(-A + \frac{k}{2}I + 4\pi^2\xi^2 I\right) B_{\xi}^{\alpha-1} + 128 \alpha \pi^4\xi^2 B_\xi^{\alpha-1} \\ \ecart
&& \dis + 256\alpha(\alpha-1)\pi^4\xi^2\left(-A + \dfrac{k}{2}I + 4\pi^2\xi^2 I\right)^2 B_{\xi}^{\alpha-2}.
\end{array}$$
\end{Lem}

\begin{proof}
From \refL{Lem B_xi BIP inversible}, it follows that $B_\xi^\alpha \in \L(X)$ is well defined.
Since $0\in \rho(A)$, there exists $c_0 >0$, such that $\overline{B(0,c_0)} \subset \rho(A)$, such that
$$A^{-1}B_\xi^\alpha = \frac{1}{2i\pi} \int_{\Gamma_A} \left(z\left(z+k+8\pi^2\xi^2\right)+\left(\frac{k}{2} + 4\pi^2\xi^2\right)^2 + r'\right)^\alpha z^{-1}\left(zI - A \right)^{-1} dz,$$
where $\Gamma_A$ is a sectorial curve, positively oriented, of angle $\theta'_A$ which avoids 0 at distance $c_0$, where $\theta'_A $ is fixed in $\dis \left(\theta_A, \frac{\pi}{2} \right)$. Thus
$$\begin{array}{lll}
\dis A^{-1}\frac{\partial B_{\xi}^\alpha }{\partial \xi} & = & \dis  \frac{\partial}{\partial \xi}\left(\frac{1}{2i\pi} \int_{\Gamma_A} \left(z\left(z+k+8\pi^2\xi^2\right)+\left(\frac{k}{2} + 4\pi^2\xi^2\right)^2 + r'\right)^{\alpha} z^{-1}\left(zI - A \right)^{-1} dz \right) \\ \ecart

& = & \dis \frac{1}{2i\pi} \int_{\Gamma_A}  \frac{\partial}{\partial \xi}\left(z\left(z+k+8\pi^2\xi^2\right)+\left(\frac{k}{2} + 4\pi^2\xi^2\right)^2 + r'\right)^{\alpha} z^{-1}\left(zI - A \right)^{-1} dz \\ \ecart

& = & \dis \frac{16\alpha\pi^2\xi}{2i\pi} \int_{\Gamma_A} \dfrac{z + \dfrac{k}{2} + 4\pi^2\xi^2}{\left(z\left(z+k+8\pi^2\xi^2\right)+\left(\dfrac{k}{2} + 4\pi^2\xi^2\right)^2 + r'\right)^{1-\alpha}} z^{-1}\left(zI - A \right)^{-1} dz \\ \ecart

& = & \dis 16\alpha\pi^2\xi\left(-A + \frac{k}{2}I + 4\pi^2\xi^2 I\right)A^{-1}B_\xi^{\alpha-1},
\end{array}$$
hence
$$\frac{\partial B_{\xi}^\alpha }{\partial \xi} = 16\alpha\pi^2\xi\left(-A + \frac{k}{2}I + 4\pi^2\xi^2 I\right) B_{\xi}^{\alpha-1}.$$
In the same way, we have
$$\frac{\partial^2 B_{\xi}^\alpha }{\partial \xi^2} = 16\alpha\pi^2\left(-A + \frac{k}{2}I + 4\pi^2\xi^2 I\right) B_{\xi}^{\alpha-1} + 16\alpha\pi^2\xi\frac{\partial }{\partial \xi}\left(\left(-A + \dfrac{k}{2}I + 4\pi^2\xi^2 I\right) B_{\xi}^{\alpha-1}\right),$$
and setting $\delta_\alpha(\xi) = \dfrac{\partial}{\partial \xi}\left(\left(-A + \dfrac{k}{2}I + 4\pi^2\xi^2 I\right) B_{\xi}^{\alpha-1}\right)$, we obtain that
$$\begin{array}{lll}
A^{-1}\delta_\alpha(\xi) & = & \dis \frac{1}{2i\pi} \int_{\Gamma_A} \frac{\partial}{\partial \xi}\left(\dfrac{z + \dfrac{k}{2} + 4\pi^2\xi^2}{\left(z\left(z+k+8\pi^2\xi^2\right)+\left(\dfrac{k}{2} + 4\pi^2\xi^2\right)^2 + r'\right)^{1-\alpha}}\right) z^{-1}\left(zI - A \right)^{-1} dz \\ \\

& = & \dis \frac{1}{2i\pi} \int_{\Gamma_A} 8\pi^2\xi \left(z\left(z+k+8\pi^2\xi^2\right)+\left(\dfrac{k}{2} + 4\pi^2\xi^2\right)^2 + r'\right)^{\alpha-1} z^{-1}\left(zI - A \right)^{-1} dz \\ \ecart

&& + \dis \frac{\alpha-1}{2i\pi} \int_{\Gamma_A} \dfrac{\left(z + \dfrac{k}{2} + 4\pi^2\xi^2\right)\left(16\pi^2 \xi z + 8k\pi^2\xi + 64\pi^4\xi^3 \right)}{\left(z\left(z+k+8\pi^2\xi^2\right)+\left(\dfrac{k}{2} + 4\pi^2\xi^2\right)^2 + r'\right)^{2-\alpha}} z^{-1}\left(zI - A \right)^{-1} dz \\ \\

& = & \dis 8\pi^2\xi A^{-1} B_\xi^{\alpha-1}  \\ \ecart
&& \dis + \frac{(\alpha-1)16\pi^2\xi}{2\pi} \int_{\Gamma_A} \dfrac{\left(z + \dfrac{k}{2} + 4\pi^2\xi^2\right)^2}{\left(z\left(z+k+8\pi^2\xi^2\right)+\left(\dfrac{k}{2} + 4\pi^2\xi^2\right)^2 + r'\right)^{2-\alpha}} z^{-1}\left(zI - A \right)^{-1} dz \\ \\

& = & \dis  8\pi^2\xi A^{-1} B_\xi^{\alpha-1} + (\alpha-1)16\pi^2\xi\left(-A + \dfrac{k}{2}I + 4\pi^2\xi^2 I\right)^2 A^{-1} B_{\xi}^{\alpha-2},
\end{array}$$
hence
\begin{equation}\label{delta_alpha}
\delta_\alpha(\xi) = 8\pi^2\xi B_\xi^{\alpha-1} + (\alpha-1)16\pi^2\xi\left(-A + \dfrac{k}{2}I + 4\pi^2\xi^2 I\right)^2 B_{\xi}^{\alpha-2},
\end{equation}
and
$$\frac{\partial^2 B_{\xi}^\alpha }{\partial \xi^2} = 16\alpha\pi^2\left(-A + \frac{k}{2}I + 4\pi^2\xi^2 I\right) B_{\xi}^{\alpha-1} + 16\alpha\pi^2\xi \delta_\alpha(\xi).$$
\end{proof}

\begin{Prop} \label{Prop m(xi) et xi m'(xi) R-bornés}
Let $r\in\RR$. Assume that ($H_1$), ($H_2$), ($H_3$) and ($H_4$) hold. The set of bounded operators
$$\left\{B_\xi^{ir} : \xi \in \RR \setminus \{0\}\right\} \quad \text{and} \quad \left\{\xi \frac{\partial B_\xi^{ir}}{\partial \xi} : \xi \in \RR \setminus \{0\}\right\},$$
where $B_\xi$ is defined in \eqref{Def B_xi}, are $\mathcal{R}$-bounded.
\end{Prop}
\begin{proof}
Let $\xi \in \RR\setminus \{0\}$. From ($H_3$), ($H_4$) and Theorem 2.3, p. 69 in \cite{arendt-bu-haase}, we have 
$$-A + \frac{k}{2}I + 4\pi^2 \xi^2 I \in \text{BIP}(X,\theta_A)\quad \text{and} \quad 0 \in \rho\left(-A + \frac{k}{2}I + 4\pi^2 \xi^2 I \right),$$ 
and due to \refL{Lem dérivée B_xi^alpha}, it follows that
$$\frac{\partial B_\xi^{ir}}{\partial \xi} = 16 ir\pi^2 \xi \left(-A + \frac{k}{2} I + 4 \pi^2\xi^2 I\right)B_\xi^{-1} B_\xi^{ir} = f_\xi\left(-A+\frac{k}{2}I + 4\pi^2\xi^2 I\right) B_\xi^{ir}.$$
where
$$f_\xi(z) = \frac{16ir\pi^2 \xi z}{\left(z-\frac{k}{2} - 4\pi^2\xi^2 \right)\left(z+\frac{k}{2} + 4\pi^2\xi^2 + \left(\frac{k}{2} + 4\pi^2\xi^2\right)^2 + r'\right)}, \quad z \in \overline{S_{\theta_A'}},~ \theta_A' \in \left(\theta_A,\frac{\pi}{2}\right).$$
Note that if $k<0$ and $\xi = \frac{\sqrt{|k|}}{2\pi\sqrt{2}}$, then $\frac{k}{2} + 4\pi^2\xi^2 = 0$. Thus, for all $k \in \RR$, we have to consider
$$\int_0^{+\infty} \left\|\frac{\partial B_\xi^{ir}}{\partial \xi}\right\|_{\L(X)} d\xi = \int_0^{1+\frac{\sqrt{|k|}}{2\pi\sqrt{2}}} \left\|\frac{\partial B_\xi^{ir}}{\partial \xi}\right\|_{\L(X)} d\xi + \int_{1+\frac{\sqrt{|k|}}{2\pi\sqrt{2}}}^{+\infty} \left\|\frac{\partial B_\xi^{ir}}{\partial \xi}\right\|_{\L(X)} d\xi.$$
Moreover, due to \refL{Lem B_xi BIP inversible}, there exists $C>0$, independant of $r$ and $\xi$, such that
$$\left\|\frac{\partial B_\xi^{ir}}{\partial \xi}\right\|_{\L(X)} \leqslant Ce^{2\theta_A |r|} \left\|\int_{\Gamma_{A}} f_\xi(z) \left(-A+\frac{k}{2}I + 4\pi^2\xi^2 I - zI\right)^{-1} dz\right\|_{\L(X)},$$
where the path $\Gamma_{A}$ is the boundary positively oriented of $S_{\theta_A'}$, with $\theta_A'$ fixed in $\left(\theta_A, \frac{\pi}{2}\right)$. Thus there exists $C_A>0$, independant of $r$ and $\xi$, such that
$$\begin{array}{lll}
\dis \left\|\frac{\partial B_\xi^{ir}}{\partial \xi}\right\|_{\L(X)} &\leqslant & \dis Ce^{2\theta_A |r|}\left\|\int_0^{+\infty} f_\xi\left(|z|e^{i\theta_A}\right) \left(-A+\frac{k}{2}I + 4\pi^2\xi^2 I - |z|e^{i\theta_A}I\right)^{-1} e^{i\theta_A}d|z|\right\|_{\L(X)} \\ \ecart
&& \dis + Ce^{2\theta_A |r|}\left\| \int_0^{+\infty} f_\xi\left(|z|e^{-i\theta_A}\right) \left(-A+\frac{k}{2}I + 4\pi^2\xi^2 I - |z|e^{-i\theta_A}I\right)^{-1} e^{-i\theta_A}d|z|\right\|_{\L(X)} \\ \\

& \leqslant & \dis Ce^{2\theta_A |r|}\int_0^{+\infty} \left|f_\xi(|z|e^{i\theta_A})\right| \frac{C_A}{1+|z|}~ d|z| + Ce^{2\theta_A |r|}\int_0^{+\infty} \left|f_\xi(|z|e^{-i\theta_A})\right| \frac{C_A}{1+|z|}~ d|z|.
\end{array}$$
From \refL{Lem Lab-Ter}, which remains true for $z=0$, there exist $K_0,K_1 >0$, independant of $r$ and $\xi$, such that
$$\left|f_\xi \left(|z|e^{i\theta_A}\right)\right| \leqslant \frac{K_0 \,\xi |z|}{|z|\,|z+r'|} = \frac{K_0 \,\xi}{|z+r'|} \leqslant \frac{K_1 \,\xi}{|z|+r'},$$
and
$$\left|f_\xi \left(|z|e^{-i\theta_A}\right)\right| \leqslant \frac{K_0\, \xi |z|}{|z|\, |z+r'|} = \frac{K_0 \,\xi}{|z+r'|}\leqslant \frac{K_1 \,\xi}{|z|+r'}.$$
Then, since $r'>0$, there exists $K_2>0$, independant of $\xi$, such that
$$\left\|\frac{\partial B_\xi^{ir}}{\partial \xi}\right\|_{\L(X)} \leqslant \xi \int_0^{+\infty} \frac{2 Ce^{2\theta_A |r|}C_A \,K_1}{(|z|+r')(1+|z|)}~d|z| \leqslant \xi \int_0^{+\infty} \frac{2Ce^{2\theta_A |r|}C_A \,K_1}{\left(|z|+\min(r',1)\right)^2}~d|z| \leqslant K_2 \,\xi.$$
Hence
$$\int_0^{1+\frac{\sqrt{|k|}}{2\pi\sqrt{2}}} \left\|\frac{\partial B_\xi^{ir}}{\partial \xi}\right\|_{\L(X)} d\xi \leqslant K_2\int_0^{1+\frac{\sqrt{|k|}}{2\pi\sqrt{2}}} \xi~d\xi = \frac{K_2}{2} \left(1+\frac{\sqrt{|k|}}{2\pi\sqrt{2}}\right)^2,$$
and there exists $K_3 > 0$, independant of $\xi$, such that
$$ \int_{1+\frac{\sqrt{|k|}}{2\pi\sqrt{2}}}^{+\infty} \left\|\frac{\partial B_\xi^{ir}}{\partial \xi}\right\|_{\L(X)} d\xi \leqslant \int_{1+\frac{\sqrt{|k|}}{2\pi\sqrt{2}}}^{+\infty} \int_0^{+\infty} \left( \left|f_\xi(|z|e^{i\theta_A})\right| + \left|f_\xi(|z|e^{-i\theta_A})\right|\right)\frac{K_3}{1+|z|}~ d|z|\, d\xi.$$
Moreover, we set 
$$C_\xi = \frac{k}{2} + 4 \pi^2 \xi^2.$$
Note that when $\xi \geqslant 1+\frac{\sqrt{|k|}}{2\pi\sqrt{2}}>0$, then $C_\xi > 0$. Thus, from \refL{Lem Lab-Ter}, which remains true for $z=0$, there exist $K_4,K_5,K_6 > 0$, independant of $\xi$, such that
$$\begin{array}{lll}
\dis \int_0^{+\infty} \left|f_\xi(|z|e^{i\theta_A})\right|\frac{K_3}{1+|z|}~ d|z| &=& \dis \int_0^1 \frac{16 r \pi^2 \xi |z|}{\left||z|e^{i\theta_A} - C_\xi\right|\left| |z|e^{i\theta_A} + C_\xi + C_\xi^2 + r'\right|} \frac{K_3}{1+|z|}~d|z| \\ \ecart
&& \dis + \int_1^{+\infty} \frac{16 r \pi^2 \xi |z|}{\left||z|e^{i\theta_A} - C_\xi\right|\left| |z|e^{i\theta_A} + C_\xi + C_\xi^2 + r'\right|} \frac{K_3}{1+|z|}~d|z| \\ \\

& \leqslant & \dis \int_0^1 \frac{K_4\, \xi}{C_\xi^3} ~d|z| + \int_1^{+\infty} \frac{K_5 \,\xi}{\left||z|e^{i\theta_A} - C_\xi\right|\left| |z|e^{i\theta_A} + C_\xi + C_\xi^2 \right|} ~d|z| \\ \ecart

& \leqslant & \dis \frac{K_4 \,\xi}{C_\xi^3} + \int_1^{+\infty} \frac{K_5 \,\xi}{\left||z|^2e^{2i\theta_A} - C_\xi^2 + C_\xi^2 |z|e^{i\theta_A} - C_\xi^3\right|}~d|z| \\ \ecart

& \leqslant & \dis \frac{K_4\, \xi}{C_\xi^3} + \int_1^{+\infty} \frac{K_6\,\xi}{\left||z|^2e^{2i\theta_A} + C_\xi^2 |z|e^{i\theta_A} - C_\xi^3\right|}~d|z|.
\end{array}$$
From Proposition 13, p. 8 in \cite{DLMMT}, we have
$$\left|\arg\left(|z|^2e^{2i\theta_A} + C_\xi^2 |z|e^{i\theta_A} \right)\right| \leqslant \max(2\theta_A,\theta_A) = 2\theta_A,$$
and due to \refP{Prop DoreLabbas}, we have
$$\begin{array}{lll}
\dis \left||z|^2e^{2i\theta_A} + C_\xi^2 |z|e^{i\theta_A} - C_\xi^3\right| & \geqslant & \dis \left(\left||z|^2e^{2i\theta_A} + C_\xi^2 |z|e^{i\theta_A}\right| + C_\xi^3\right) \cos\left(\frac{2\theta_A}{2} \right) \\ \ecart
& \geqslant & \dis \left(\left(\left||z|^2e^{2i\theta_A}\right| + \left| C_\xi^2 |z|e^{i\theta_A}\right|\right) \cos\left(\frac{2\theta_A - \theta_A}{2}\right) + C_\xi^3 \right)\cos(\theta_A) \\ \ecart
& \geqslant & \dis \left(|z|^2\cos\left(\frac{\theta_A}{2}\right) + C_\xi^3 \right)\cos(\theta_A) \\ \ecart
& \geqslant & \dis \left(|z|^2 + C_\xi^3\right)\cos\left(\frac{\theta_A}{2}\right)\cos(\theta_A) > 0.
\end{array}$$
Thus, there exists $K_7 >0$, independant of $\xi$, such that
$$\begin{array}{lll}
\dis \int_0^{+\infty} \left|f_\xi(|z|e^{i\theta_A})\right|\frac{K_3}{1+|z|}~ d|z| &\leqslant & \dis \frac{K_4\, \xi}{C_\xi^3} + \int_1^{+\infty} \frac{K_7\, \xi}{|z|^2 + C_\xi^3}~d|z| \\ \ecart

& \leqslant & \dis \frac{K_4\, \xi}{C_\xi^3} + \frac{K_7\,\xi}{C_\xi^3} \int_{\frac{1}{C_\xi^{3/2}}}^{+\infty} \frac{C_\xi^{3/2}}{x^2 + 1}~dx \\ \ecart

& \leqslant & \dis \frac{K_4\, \xi}{C_\xi^3} + \frac{K_7\,\xi}{C_\xi^{3/2}} \left(\frac{\pi}{2} - \arctan\left(\frac{1}{C_\xi^{3/2}} \right)\right) \\ \ecart

& \leqslant & \dis \frac{K_4\, \xi}{C_\xi^3} + \frac{K_7\,\pi\xi}{2 \,C_\xi^{3/2}}
\end{array}$$
In the same way, we obtain
$$\int_0^{+\infty} \left|f_\xi(|z|e^{-i\theta_A})\right|\frac{K_3}{1+|z|}~ d|z| \leqslant \frac{K_4\, \xi}{C_\xi^3} + \frac{K_7\, \pi \xi}{2\,C_\xi^{3/2}}.$$
Finally, since $\xi \geqslant 1+\frac{\sqrt{|k|}}{2\pi\sqrt{2}}>0$, we have
$$0< \frac{\xi^2}{C_\xi} = \frac{\xi^2}{\frac{k}{2} + 4\pi^2\xi^2} = \frac{1}{\frac{k}{2\xi^2} + 4\pi^2} \leqslant \frac{1}{4\pi^2\left( 1 + \frac{k}{|k| + 4\pi\sqrt{2|k|} + 8\pi^2}\right)} := C_k,$$
it follows that
$$\frac{1}{C_\xi} \leqslant \frac{C_k}{\xi^2},$$
hence, there exists $K_8 > 0$, independant of $\xi$, such that
$$\begin{array}{lll}
\dis \int_{1+\frac{\sqrt{|k|}}{2\pi\sqrt{2}}}^{+\infty} \left\|\frac{\partial B_\xi^{ir}}{\partial \xi}\right\|_{\L(X)} d\xi & \leqslant & \dis \int_{1+\frac{\sqrt{|k|}}{2\pi\sqrt{2}}}^{+\infty} \frac{K_4\, \xi}{C_\xi^3} + \frac{K_7\,\pi\xi}{2 \,C_\xi^{3/2}} ~d\xi \\ \ecart

& \leqslant & \dis \int_{1+\frac{\sqrt{|k|}}{2\pi\sqrt{2}}}^{+\infty} \frac{C_k^3 K_4\, \xi}{\xi^6} + \frac{C_k^{3/2} K_7\, \pi\xi}{2 \xi^3} ~d\xi \\ \ecart

& \leqslant & \dis \int_{1}^{+\infty} \frac{K_8}{\xi^2}~d\xi = K_8.
\end{array}$$
Thus
$$\int_0^{+\infty} \left\|\frac{\partial B_\xi^{ir}}{\partial \xi}\right\|_{\L(X)} d\xi < + \infty,$$
and similarly, we obtain
$$\int_{-\infty}^0 \left\|\frac{\partial B_\xi^{ir}}{\partial \xi}\right\|_{\L(X)} d\xi < + \infty.$$
Therefore, from Proposition 2.5, p. 739 in \cite{weis}, it follows that $\left\{B_\xi^{ir} : \xi \in \RR \setminus \{0\}\right\}$, is $\mathcal{R}$-bounded.

Now, following the same steps, we consider
$$\frac{\partial}{\partial \xi} \left(\xi \frac{\partial B_\xi^{ir}}{\partial \xi}\right) = \frac{\partial B_\xi^{ir}}{\partial \xi} + \xi \frac{\partial^2 B_\xi^{ir}}{\partial \xi^2}.$$
Thus, we have to show that
$$\int_0^{+\infty} \left\|\xi\frac{\partial^2 B_\xi^{ir}}{\partial \xi^2}\right\|_{\L(X)} d\xi < + \infty.$$
From the previous steps, we have proved that there exists $K_9 >0$ such that 
\begin{equation}\label{norme xi}
\left\|\xi \left(-A + \frac{k}{2} I + 4 \pi^2\xi^2 I\right)B_\xi^{-1} \right\|_{\L(X)} \leqslant \left\{\begin{array}{ll}
K_9\, \xi & \text{if } 0 \leqslant \xi \leqslant 1 + \frac{\sqrt{|k|}}{2\pi \sqrt{2}} \\ \ecart

\dfrac{K_9}{\xi^2} & \text{if } \xi > 1 + \frac{\sqrt{|k|}}{2\pi \sqrt{2}}.
\end{array}\right.
\end{equation}
Moreover, from \refL{Lem B_xi BIP inversible} and its proof, $0 \in \rho(B_\xi)$, $-A(-A+k I+8\pi^2\xi^2 I)+r' \in \text{Sect}(2\theta_A)$ and $0 \in \rho(-A(-A+kI+8\pi^2\xi^2I)+r')$. Then, there exists $K_{10},K_{11} > 0$, independant of $\xi$, such that
\begin{equation}\label{norme xi 2}
\dis \left\|B_\xi^{-1} \right\|_{\L(X)} \leqslant \dfrac{K_{10}}{1+C_\xi^2} \leqslant \left\{\begin{array}{ll}
K_{10} & \text{if } 0 \leqslant \xi \leqslant 1 + \frac{\sqrt{|k|}}{2\pi \sqrt{2}} \\ \ecart

\dfrac{K_{11}}{\xi^4} & \text{if } \xi > 1 + \frac{\sqrt{|k|}}{2\pi \sqrt{2}}.
\end{array}\right.
\end{equation}
From \refL{Lem dérivée B_xi^alpha}, we have
$$\begin{array}{lll}
\dis \frac{\partial^2 B_{\xi}^{ir} }{\partial \xi^2} & = & \dis 16ir\pi^2\left(-A + \frac{k}{2}I + 4\pi^2\xi^2 I\right) B_{\xi}^{-1}B_{\xi}^{ir} + 128 ir \pi^4\xi^2 B_\xi^{-1} B_{\xi}^{ir}\\ \ecart
&& \dis + 256ir(ir-1)\pi^4\xi^2\left(-A + \dfrac{k}{2}I + 4\pi^2\xi^2 I\right)^2 B_{\xi}^{-2}B_{\xi}^{ir}, 
\end{array}$$
and due to \refL{Lem B_xi BIP inversible}, there exists $K_{12}>0$, independant of $\xi$, such that
$$\begin{array}{lll}
\dis \left\|\xi \frac{\partial^2 B_{\xi}^{ir} }{\partial \xi^2} \right\|_{\L(X)} & \leqslant & \dis 16|r|\pi^2\left\|\xi\left(-A + \frac{k}{2}I + 4\pi^2\xi^2 I\right) B_{\xi}^{-1}\right\|_{\L(X)} \left\|B_{\xi}^{ir}\right\|_{\L(X)} \\ \ecart
&& \dis + 128 |r| \pi^4 \left\|\xi^2 B_\xi^{-1}\right\|_{\L(X)} \left\|B_{\xi}^{ir}\right\|_{\L(X)} \\ \ecart
&& \dis + 256|r| (|r|+1) \pi^4 \left\|\xi\left(-A + \dfrac{k}{2}I + 4\pi^2\xi^2 I\right) B_{\xi}^{-1}\right\|^2_{\L(X)} \left\|B_{\xi}^{ir}\right\|_{\L(X)} \\ \\

& \leqslant & \dis K_{12} \left\|\xi\left(-A + \dfrac{k}{2}I + 4\pi^2\xi^2 I\right) B_{\xi}^{-1}\right\|_{\L(X)} + K_{12} \,\xi^2 \left\| B_\xi^{-1}\right\|_{\L(X)} \\ \ecart
&& \dis + K_{12} \left\|\xi\left(-A + \dfrac{k}{2}I + 4\pi^2\xi^2 I\right) B_{\xi}^{-1}\right\|^2_{\L(X)}.
\end{array}$$
Then, from \eqref{norme xi} and \eqref{norme xi 2}, there exists $K_{13},K_{14}>0$, independant of $\xi$, such that
$$\begin{array}{lll}
\dis \int_0^{+\infty} \left\|\xi \frac{\partial^2 B_{\xi}^{ir} }{\partial \xi^2} \right\|_{\L(X)}d\xi & = & \dis \int_0^{1 + \frac{\sqrt{|k|}}{2\pi \sqrt{2}}} \left\|\xi \frac{\partial^2 B_{\xi}^{ir} }{\partial \xi^2} \right\|_{\L(X)}d\xi + \int_{1 + \frac{\sqrt{|k|}}{2\pi \sqrt{2}}}^{+\infty} \left\|\xi \frac{\partial^2 B_{\xi}^{ir} }{\partial \xi^2} \right\|_{\L(X)}d\xi \\ \ecart

& \leqslant & \dis K_{13} \int_0^{1 + \frac{\sqrt{|k|}}{2\pi \sqrt{2}}} \xi + 2\xi^2 ~d\xi + K_{13} \int_{1 + \frac{\sqrt{|k|}}{2\pi \sqrt{2}}}^{+\infty} \frac{2}{\xi^2} + \frac{1}{\xi^4}~d\xi \\ \ecart

& \leqslant & \dis K_{14} \left(1 + \int_1^{+\infty} \frac{3}{\xi^2}~d\xi \right) = 4 K_{14}.
\end{array}$$
Thus
$$\int_0^{+\infty} \left\|\frac{\partial}{\partial \xi} \left(\xi \frac{\partial B_\xi^{ir}}{\partial \xi}\right)\right\|_{\L(X)}d\xi < + \infty,$$
and similarly
$$\int_{-\infty}^0 \left\|\frac{\partial}{\partial \xi} \left(\xi \frac{\partial B_\xi^{ir}}{\partial \xi}\right)\right\|_{\L(X)}d\xi < + \infty.$$
Therefore, from Proposition 2.5, p. 739 in \cite{weis}, it follows that $\left\{\xi \frac{\partial B_\xi^{ir}}{\partial \xi} : \xi \in \RR \setminus \{0\}\right\}$, is $\mathcal{R}$-bounded.
\end{proof}

\section{Proofs of main results}\label{Sect Proofs of main results}

\subsection{Proof of \refT{Th BIP}}

\begin{proof}
Let $\theta_A \in [0,\pi/2)$. From \cite{LMT}, Proposition 4.1, p. 935, for $i=1,2$, we deduce that 
$$-\A_i + \frac{k^2}{4} I \in \text{Sect\,}(2\theta_A),$$ 
for $i=3,4$, there exists $r'>0$, for all $\theta_0 > 0$, such that
$$\left\{\begin{array}{ll}
\dis -\A_i + \frac{k^2}{4}I + r' I \in \text{Sect\,}(2\theta_A), & \text{if } \dis \theta_A > 0,\\ \ecart

\dis -\A_i + \frac{k^2}{4}I + r' I \in \text{Sect\,}(\theta_0), & \text{if } \dis \theta_A = 0.
\end{array}\right.$$
Thus, for $i=1,2,3,4$, we set
$$-A_i = -\A_i + \frac{k^2}{4}I + r' I.$$
From \cite{LMT}, Proposition 4.9 and Proposition 4.10, p. 942, for $i=1,2,3,4$, it follows that
\begin{equation}\label{0 in rho(Ai)}
0\in \rho\left(-A_i\right).
\end{equation}
Thus operators $-A_i$ are sectorial injective operators. Moreover, due to $(H_1)$, since $X$ is a UMD space, then $X$ is a reflexive space and from \cite{haase}, Proposition~2.1.1, h), p. 20-21, we have
$$\overline{D(A_i)} = \overline{R(A_i)} = X.$$
Now, it remains to prove that the imaginary powers of $-A_i$ are bounded.

Let $\varepsilon >0$, $r \in \RR$ and $\lambda \in \rho(-A_i)$. Using the Dunford-Riesz integral, we have
$$\left[\left(-A_i\right)^{-\varepsilon + ir} f\right](x) = \frac{1}{2i\pi} \int_\Gamma \lambda^{-\varepsilon+ir} \left[\left(-A_i - \lambda I\right)^{-1} f\right](x)~ d\lambda,$$
where $\Gamma$ is a sectorial curve such that for $\varepsilon_0 > 0$ fixed and for all $\theta' > 0$, we have
$$\Gamma = \left\{\begin{array}{ll}
\dis \partial \left(\overline{S_{2\theta_A + \theta'}} \setminus B(0,\varepsilon_0)\right), & \text{if } \dis \theta_A > 0,\\ \ecart

\dis \partial \left(\overline{S_{\theta_0 + \theta'}}\setminus B(0,\varepsilon_0)\right), & \text{if } \dis \theta_A = 0.
\end{array}\right.$$
Let $\dis \lambda \in \rho(-A_i) = \rho\left(-\A_i + \frac{k^2}{4} I + r' I\right)$. Then, we have 
$$\left(-A_i - \lambda I \right)^{-1} = \left(-\A_i - \left(\lambda - \frac{k^2}{4} - r'\right)I\right)^{-1} \in \L(X).$$
To simplify the notations, we set  
$$\mu = \lambda - \frac{k^2}{4} - r',$$
it follows that
$$\left(-A_i - \lambda I \right)^{-1} = \left(-\A_i - \mu I\right)^{-1} \in \L(X).$$
Let $c=b-a>0$, $\varphi \in L^p(a,b;X)$ and $T_\mu$ be a linear operator such that $0 \in \rho(T_\mu)$ and $-T_\mu \in$ Sect$(\theta_T)$, with $\theta_T \in [0,\pi/2)$. In order to simplify the notations, for all $x \in [a,b]$, we set
\begin{equation*}
\begin{array}{rll}
K_{T_\mu,\varphi}(x) & = & \dis \frac{1}{2} \left(e^{(b-x)T_\mu} e^{c T_\mu} - e^{(x-a)T_\mu}\right)  \left(I - e^{2c T_\mu}\right)^{-1} T_\mu^{-1} \int_a^b e^{(s-a)T_\mu} \varphi(s)~ ds \\ \ecart
&&+ \dis \frac{1}{2} \left(e^{(x-a)T_\mu} e^{c T_\mu} - e^{(b-x)T_\mu}\right) \left(I - e^{2c T_\mu}\right)^{-1} T_\mu^{-1} \int_a^b e^{(b-s)T_\mu} \varphi(s)~ ds
\end{array}
\end{equation*}
and
\begin{equation*}
\begin{array}{lll}
J_{T_\mu,\varphi}(x)& = & \dis \frac{1}{2} \, T_\mu^{-1} \int_a^x e^{(x-s)T_\mu} \varphi(s)~ ds + \frac{1}{2} \, T_\mu^{-1} \int_x^b e^{(s-x)T_\mu} \varphi(s)~ ds \\ \ecart
& = & \dis \frac{1}{2} \, T_\mu^{-1} \int_a^b e^{|x-s|T_\mu} \varphi(s)~ ds
\end{array}
\end{equation*}
From the representation formulas in \cite{LMT}, p. 952-953-954-955 and 957, we deduce that
$$\begin{array}{lll}
\dis \left(-A_i - \lambda I \right)^{-1}f(x) & = & \dis \left(e^{(x-a)M_\mu} - e^{(b-x)M_\mu}\right) \alpha_{1,\mu,i} + \left(e^{(x-a)L_\mu} - e^{(b-x)L_\mu} \right) \alpha_{2,\mu,i} \\ \ecart
&& \dis + \left(e^{(x-a)M_\mu} + e^{(b-x)M_\mu}\right) \alpha_{3,\mu,i} + \left(e^{(x-a)L_\mu} + e^{(b-x)L_\mu} \right) \alpha_{4,\mu,i} + F_{0,f}(x),
\end{array}
$$
where, due to (53) in \cite{LMT}, we have
$$F_{0,f}(x) = K_{M_\mu, K_{L_\mu, f} +J_{L_\mu , f}} (x) + J_{M_\mu, K_{L_\mu, f} +J_{L_\mu , f}} (x),$$
with
$$M_\mu = -\sqrt{- A + \frac{k}{2}I - i\sqrt{-\mu - \frac{k^2}{4}}I} \quad \text{and} \quad L_\mu = -\sqrt{- A + \frac{k}{2} + i\sqrt{-\mu - \frac{k^2}{4}}I}.$$
Moreover, we recall for the boundary conditions (BC1), from \cite{LMT}, p. 952, that
$$\alpha_{1,\mu,1} = \alpha_{2,\mu,1} = \alpha_{3,\mu,1} = \alpha_{4,\mu,1} = 0,$$
for the boundary conditions (BC2), from \cite{LMT}, p. 953, we recall that
\begin{equation*}
\left\{\begin{array}{rcl}
\alpha_{1,\mu,2} & = & \dis - \frac{1}{2} \left(I + e^{cM_\mu}\right)^{-1} M_\mu^{-1} \left(F_{0,f}'(a) + F'_{0,f}(b)\right) \\ \ecart 

\alpha_{2,\mu,2}  & = & \dis 0 \\ \ecart

\alpha_{3,\mu,2}  & = & \dis - \frac{1}{2} \left(I - e^{cM_\mu}\right)^{-1} M_\mu^{-1} \left(F_{0,f}'(a) - F'_{0,f}(b)\right) \\ \ecart 

\alpha_{4,\mu,2}  & = & \dis 0,
\end{array}\right.
\end{equation*}
for the boundary conditions (BC3), from \cite{LMT}, p. 955, we recall that
\begin{equation*}
\left\{\begin{array}{rcl}
\alpha_{1,\mu,3} & = & \dis~~ \frac{1}{2} B_\mu^{-1} (L_\mu + M_\mu) U_\mu^{-1} \left(I - e^{cL_\mu}\right) \left(F_{0,f}'(a) + F'_{0,f}(b)\right) \\ \ecart

\alpha_{2,\mu,3} & = & \dis -\frac{1}{2} B_\mu^{-1} (L_\mu + M_\mu) U_\mu^{-1} \left(I - e^{cM_\mu}\right) \left(F_{0,f}'(a) + F'_{0,f}(b)\right) \\ \ecart

\alpha_{3,\mu,3} & = & \dis ~~\frac{1}{2} B_\mu^{-1} (L_\mu + M_\mu) V_\mu^{-1} \left(I + e^{cL_\mu}\right) \left(F_{0,f}'(a) - F'_{0,f}(b)\right)  \\ \ecart

\alpha_{4,\mu,3} & = & \dis -\frac{1}{2} B_\mu^{-1} (L_\mu + M_\mu) V_\mu^{-1} \left(I + e^{cM_\mu}\right) \left(F_{0,f}'(a) - F'_{0,f}(b)\right),
\end{array}\right.
\end{equation*}
where
\begin{equation*}
\left\{\begin{array}{llll}
B_\mu &=&\dis 2i\sqrt{-\mu-\frac{k^2}{4}}I \\ \ecart

U_\mu &:=& I-e^{c \left( L_\mu + M_\mu \right) } - B_\mu^{-1} \left( L_\mu + M_\mu\right)^2 \left( e^{c M_\mu}-e^{c L_\mu}\right) \\ \ecart

V_\mu &:=& I-e^{c \left( L_\mu + M_\mu\right) } + B_\mu^{-1}\left( L_\mu + M_\mu\right)^2 \left( e^{c M_\mu}-e^{c L_\mu}\right),
\end{array}\right.
\end{equation*}
and for the boundary conditions (BC4), from \cite{LMT}, p. 957, we recall that
\begin{equation*}
\left\{\begin{array}{rcl}
\alpha_{1,\mu,4} & = & \dis - \frac{1}{2} B_\mu^{-1} (L_\mu + M_\mu) V_\mu^{-1} \left(I - e^{cL_\mu}\right) L_\mu  M_\mu^{-1} \left(F_{0,f}'(a) + F'_{0,f}(b)\right) \\ \ecart

\alpha_{2,\mu,4} & = & \dis ~~\frac{1}{2} B_\mu^{-1} (L_\mu + M_\mu) V_\mu^{-1} \left(I - e^{cM_\mu}\right) M_\mu L_\mu^{-1} \left(F_{0,f}'(a) + F'_{0,f}(b)\right) \\ \ecart

\alpha_{3,\mu,4} & = & \dis - \frac{1}{2} B_\mu^{-1} (L_\mu + M_\mu) U_\mu^{-1} \left(I + e^{cL_\mu}\right) L_\mu M_\mu^{-1} \left(F_{0,f}'(a) - F'_{0,f}(b)\right)  \\ \ecart

\alpha_{4,\mu,4} & = & \dis ~~\frac{1}{2} B_\mu^{-1} (L_\mu + M_\mu) U_\mu^{-1} \left(I + e^{cM_\mu}\right) M_\mu L_\mu^{-1} \left(F_{0,f}'(a) - F'_{0,f}(b)\right).
\end{array}\right.
\end{equation*}
Now, in view to apply \refL{Lem Amann}, we have to split this resolvant operator into two. To this end, we define two relations which simplify calculations.

Let $P_\lambda$, $Q_{1,\lambda}$ and $Q_{2,\lambda}$ three bounded linear operators such that
$$P_\lambda = Q_{1,\lambda} + Q_{2,\lambda},$$
where $\lambda \in \Gamma$, a sectorial curve which surrounds the spectrum of $P_\lambda$. We define the relation $\sim$ such that  
$$P_\lambda \sim Q_{1,\lambda},$$
means that
$$\lambda \longmapsto Q_{2,\lambda} \in L^1\left(\Gamma,\L(X)\right).$$
In the same way, we define the relation $\simeq$ such that for all $\varphi \in L^p(a,b;X)$, if $P_\lambda \sim Q_{1,\lambda}$, then 
$$P_\lambda \varphi(x) \simeq Q_{1,\lambda} \varphi(x).$$
For all $t_0 >0$, since we have
$$\left(I - e^{t_0 T_\mu}\right)^{-1} = I + e^{t_0 T_\mu}\left(I - e^{t_0 T_\mu}\right)^{-1},$$
from \cite{dore-yakubov}, Lemma 2.6, statement a), p. 104, it follows that
$$\left(I - e^{t_0 T_\mu}\right)^{-1} \sim I.$$
Note that we have
$$K_{T_\mu,\varphi}(a) + J_{T_\mu,\varphi}(a) = 0 = K_{T_\mu,\varphi}(b) + J_{T_\mu,\varphi}(b),$$
and when $x\in(a,b)$, from \refL{Lem T^alp e^(tT) borné}, for all $m>0$, we obtain
$$\begin{array}{lll}
\dis \left\|T_\mu^{-m}T_\mu^m\left(e^{2(x-a)T_\mu} \pm e^{2(b-x)T_\mu}\right)\right\|_{\L(X)} & \leqslant & \dis \left\|T_\mu^{-m}\right\|_{\L(X)} \left\|T_\mu^m\left(e^{2(x-a)T_\mu} \pm e^{2(b-x)T_\mu}\right)\right\|_{\L(X)} \\ \ecart
& \leqslant & \dis C\left\|T_\mu^{-m}\right\|_{\L(X)}.
\end{array}$$
Thus, for $m$ enough large, since in our case $T_\mu=M_\mu$ or $L_\mu$, it follows that
$$F_{0,f}(x) = K_{M_\mu, K_{L_\mu, f} +J_{L_\mu , f}} (x) + J_{M_\mu, K_{L_\mu, f} +J_{L_\mu , f}} (x) \simeq J_{M_\mu,J_{L_\mu,f}}(x),$$
hence, for the boundary conditions (BC1), we have
$$(-A_1 - \lambda I)^{-1}f(x) \simeq J_{M_\mu,J_{L_\mu,f}}(x).$$
and similarly, for the boundary conditions (BC2), (BC3) and (BC4), since
$$F'_{0,f}(a) + F'_{0,f}(b) \simeq \int_a^b \left(e^{(b-s)M_\mu} - e^{(s-a)M_\mu}\right) \left(K_{L_\mu,f}(s) + J_{L_\mu,f}(s)\right)ds,$$
and
$$F'_{0,f}(a) - F'_{0,f}(b) \simeq -\int_a^b \left(e^{(b-s)M_\mu} + e^{(s-a)M_\mu}\right) \left(K_{L_\mu,f}(s) + J_{L_\mu,f}(s)\right)ds,$$
we obtain that
$$(-A_i - \lambda I)^{-1}f(x) \simeq J_{M_\mu,J_{L_\mu,f}}(x),$$
for $i=2,3,4$.

Therefore, in view to apply \refL{Lem Amann}, it remains to study the following convolution term 
$$J_{M_\mu,J_{L_\mu,f}}(x) = \frac{1}{4} M_\mu^{-1}L_\mu^{-1} \int_a^b e^{|x-s|M_\mu} \int_a^b e^{|s-t|L_\mu} f(t)~ dt\, ds.$$
To this end, for $\varepsilon >0$, we set 
$$S_f(x) := \frac{1}{2i\pi} \int_\Gamma \lambda^{-\varepsilon+ir} J_{M_\mu,J_{L_\mu,f}}(x) ~ d\lambda.$$
Our aim is to write $S_f$ as follows
$$S_f(x) = \F^{-1}\left(m_\varepsilon (\xi) \F(f)\right)(x),$$
where $m_\varepsilon (\xi) \in \L(X)$.

To this end, we first consider that $f \in \D(a,b;X)$, where $\D(a,b;X)$ is the set of $C^\infty$ functions with compact support in $(a,b)$. For $T_\mu = M_\mu$ or $L_\mu$, for all $x \in \RR$, we set 
$$E_{T_\mu} (x) = e^{|x| T_\mu}.$$
We have
$$\begin{array}{lll}
S_f(x) & = & \dis \frac{1}{2i\pi}\int_\Gamma \lambda^{-\varepsilon + ir} J_{M_\mu,J_{L_\mu,f}}(x)~d\lambda \\ \\

& = & \dis \frac{1}{2i\pi}\int_\Gamma \frac{\lambda^{-\varepsilon + ir}}{2} M_\mu^{-1} \left(E_{M_\mu} \star J_{L_\mu,f}\right)(x)~d\lambda \\ \\

& = & \dis \frac{1}{2i\pi} \int_\Gamma \frac{\lambda^{-\varepsilon + ir}}{4} M_\mu^{-1}L_\mu^{-1} \left(E_{M_\mu} \star \left( E_{L_\mu} \star f\right)\right)(x)~d\lambda.
\end{array}$$
Here, the abstract convolution is well defined in virtue of \cite{amann}. Since we have $f \in \D(a,b;X)$ and $E_{T_\mu} \in C(a,b;\L(X))$, then $J_{L_\mu, f} = E_{L_\mu} \star f \in \S(\RR;X)$ and thus $E_{M_\mu} \star J_{L_\mu,f} \in \S(\RR;X)$, where $\mathcal{S}$ is the Schwartz space that is the function space of all functions whose derivatives are rapidly decreasing. Therefore, from \cite{amann}, Theorem 3.6, p. 13, we obtain
$$\begin{array}{lll}
S_f(x) & = & \dis \frac{1}{2i\pi} \int_\Gamma \lambda^{-\varepsilon + ir} \F^{-1}\left(\F\left(J_{M_\mu,J_{L_\mu,f}} \right)\right)(x)~ d\lambda  \\ \\

& = & \dis \F^{-1}\left(\frac{1}{2i\pi} \int_\Gamma \frac{\lambda^{-\varepsilon + ir}}{2} M_\mu^{-1}\F\left(E_{M_\mu} \star J_{L_\mu,f} \right)~ d\lambda \right)(x) \\ \\

& = & \dis \F^{-1}\left(\frac{1}{2i\pi} \int_\Gamma \frac{\lambda^{-\varepsilon + ir}}{2} M_\mu^{-1} \F\left(E_{M_\mu}\right) \F\left(J_{L_\mu,f} \right)~ d\lambda\right)(x) \\ \\

& = & \dis\F^{-1}\left(\frac{1}{2i\pi} \int_\Gamma \frac{\lambda^{-\varepsilon + ir}}{4} M_\mu^{-1} L_\mu^{-1} \F\left(E_{M_\mu}\right) \F\left(E_{L_\mu} \star f\right)~ d\lambda\right)(x) \\ \ecart

& = & \dis\F^{-1}\left(\frac{1}{2i\pi} \int_\Gamma \frac{\lambda^{-\varepsilon + ir}}{4} M_\mu^{-1} L_\mu^{-1} \F\left(E_{M_\mu}\right) \F\left(E_{L_\mu}\right) \F\left(f\right)~ d\lambda \right)(x) \\ \ecart

& = & \dis \F^{-1}\left(\frac{1}{2i\pi} \int_\Gamma \frac{\lambda^{-\varepsilon + ir}}{4} M_\mu^{-1} L_\mu^{-1} \F\left(E_{M_\mu}\right) \F\left(E_{L_\mu}\right)~ d\lambda~ \F\left(f\right) \right)(x).\end{array}$$
Note that, $\F(E_T)$ is the Fourier transform of the operator-valued function $E_T$ and moreover, as it is described in \cite{weis}, the following integral
$$\frac{1}{2i\pi} \int_\Gamma \frac{\lambda^{-\varepsilon + ir}}{4} M_\mu^{-1} L_\mu^{-1} \F\left(E_{M_\mu}\right) \F\left(E_{L_\mu}\right)~ d\lambda,$$
is an operator-valued Fourier multiplier.

Now, we will make explicit this Fourier multiplier. To this end, we have to determine $\F\left(E_{M_\mu}\right)$ and $\F\left(E_{L_\mu}\right)$. Using \refL{Lem F(E)}, for $T=M_\mu$ or $L_\mu$, we obtain
$$\begin{array}{lll}
\T_\F (\xi) & = & \dis \frac{\lambda^{-\varepsilon + ir}}{4} M_\mu^{-1} L_\mu^{-1} \F\left(E_{M_\mu}\right)(\xi) \F\left(E_{L_\mu}\right)(\xi) \\ \ecart

& = & \dis \lambda^{-\varepsilon + ir} M_\mu^{-1} L_\mu^{-1}  M_\mu L_\mu \left(M_\mu^2 + 4\pi^2\xi^2 I \right)^{-1} \left(L_\mu^2 + 4\pi^2\xi^2 I \right)^{-1} \\ \ecart

& = & \dis \lambda^{-\varepsilon + ir} \left(M_\mu^2 + 4\pi^2\xi^2 I \right)^{-1} \left(L_\mu^2 + 4\pi^2\xi^2 I \right)^{-1} \\ \ecart

& = & \dis \lambda^{-\varepsilon + ir} \left(M^2_\mu L^2_\mu + 4\pi^2\xi^2\left(M^2_\mu + L^2_\mu\right) + 16\pi^4\xi^4 I \right)^{-1}
\end{array}$$

and since
$$\left\{\begin{array}{l}
\dis M^2_\mu L^2_\mu = A^2 - kA - \lambda I + \frac{k^2}{4} I + r' I \\ \ecart
M^2_\mu + L^2_\mu = 2A - kI,
\end{array}\right.$$
it follows that
$$\begin{array}{lll}
\T_\F (\xi) & = & \dis \lambda^{-\varepsilon + ir} \left(A^2 - kA - \lambda I + \frac{k^2}{4} I - 4\pi^2\xi^2\left(2A - kI\right) + 16\pi^4\xi^4 I + r'I\right)^{-1} \\ \ecart

& = & \dis \lambda^{-\varepsilon + ir} \left(-A\left(-A + k I + 8\pi^2\xi^2 I\right) + \frac{k^2}{4} I + 4k\pi^2\xi^2 I + 16\pi^4\xi^4 I + r'I - \lambda I \right)^{-1} \\ \ecart

& = & \dis \lambda^{-\varepsilon + ir} \left(-A\left(-A + k I + 8\pi^2\xi^2 I\right) + \left(\frac{k}{2} + 4\pi^2\xi^2 \right)^2 I + r'I - \lambda I \right)^{-1} \\ \ecart

& = & \dis \lambda^{-\varepsilon + ir} \left(B_{\xi} - \lambda I\right)^{-1},
\end{array}$$
where 
\begin{equation}\label{B_xi}
B_{\xi} := -A\left(-A + k I + 8\pi^2\xi^2 I\right) + \left(\frac{k}{2} + 4\pi^2\xi^2 \right)^2 I + r'I.
\end{equation}
Moreover, due to \refL{Lem B_xi BIP inversible}, by functional calculus, we have
$$\begin{array}{lll}
S_f(x) & = & \dis \F^{-1}\left(\frac{1}{2i\pi} \int_\Gamma \T_\F~ d\lambda~ \F\left(f\right) \right)(x) \\ \ecart

& = & \dis \F^{-1}\left(\frac{1}{2i\pi} \int_\Gamma \lambda^{-\varepsilon + ir} \left(B_{\xi} - \lambda I\right)^{-1}~ d\lambda~ \F\left(f\right) \right)(x) \\ \ecart

& = & \dis \F^{-1}\left(m_{\varepsilon}\, \F\left(f\right) \right)(x),
\end{array}$$
where
$$m_{\varepsilon} (\xi) = B_{\xi}^{-\varepsilon + ir}.$$
Now, since $f \in \mathcal{D}(a,b;X)$, then $\F(f) \in D(B_\xi)$; due to \refL{Lem B_xi BIP inversible} and the Lebesgue's dominated convergence Theorem,
$$m(\xi) = \lim_{\varepsilon \to 0} m_\varepsilon (\xi) = \lim_{\varepsilon \to 0} B_\xi^{-\varepsilon + ir} = B_\xi^{ir}  \in \L(X).$$ 
Since $m\, \F\left(f\right) \in \mathcal{S'}(\RR;\L(X))$ that is the temperate distributions space, thus this yields
$$\lim_{\varepsilon \to 0} \F^{-1}\left(m_{\varepsilon}\, \F\left(f\right) \right) = \F^{-1}\left(m \, \F\left(f\right) \right).$$
Now, from \refP{Prop m(xi) et xi m'(xi) R-bornés}, the set  
$$\left\{m(\xi), ~\xi \in \RR\setminus\{0\}\right\} \quad \text{and} \quad \left\{\xi m'(\xi), ~\xi \in \RR\setminus\{0\}\right\},$$
are $\mathcal{R}$-bounded. Thus, applying Theorem 3.4, p. 746 in \cite{weis}, we obtain that
$$\K f = \F^{-1}\left(m \, \F\left(f\right) \right)(\cdot),$$
is extends to a bounded operator from $L^p(a,b;X)$ to $L^p(a,b;X)$, for $p\in (1,+\infty)$. Thus, there exists $C_1 >0$, independant of $\gamma$, $r$ and $\xi$, such that
$$\begin{array}{lll}
\dis \left\|\F^{-1}\left(m \, \F\left(f\right) \right)(\cdot)\right\|_{\L(X)} & \leqslant & \dis \sup_{\xi \in \RR}\left\|m (\xi) \right\|_{\L(X)} + \sup_{\xi \in \RR}\left\|\xi \,\frac{\partial m }{\partial \xi}(\xi) \right\|_{\L(X)} \\ \ecart
& \leqslant & C_1 \left( 1 + |r|\right)e^{2\theta_A |r|}  \|f\|_{L^p(a,b;X)}.
\end{array}$$
Therefore, there exists $C_2 >0$, independant of $r$ and $\varepsilon$, such that, for all $\gamma > 0$, we have
$$\left\|\F^{-1}\left(m \F(f) \right)(.)\right\|_{\L(X)} \leqslant C_2 \, e^{(2\theta_A + \gamma) |r|} \left\|f\right\|_{L^p(a,b;X)},$$
hence
$$\left\|\int_\Gamma \lambda^{-\varepsilon + ir} J_{M_\lambda,J_{L_\lambda,f}}(x)~d\lambda \right\|_{\L(X)} \leqslant C_2\, e^{(2\theta_A + \gamma) |r|} \left\|f\right\|_{L^p(a,b;X)}.$$
Finally, from \refL{Lem Amann}, for $i=1,2,3,4$, for all $\varepsilon > 0$, there exists a constant $C_{\varepsilon,p} > 0$ such that for all $r \in \RR$, we have
$$-A_i \in \text{BIP}(X,2\theta_A+\varepsilon),$$
with
$$\left\|\left(-A_i\right)^{ir}\right\|_{\L(X)} \leqslant C_{\varepsilon,p} \,e^{(2\theta_A + \varepsilon) |r|}.$$
Moreover, for $i=1,2$, due to \cite{arendt-bu-haase}, Corollary 2.4, p. 69, we obtain that
$$-\A_i + \frac{k^2}{4}I \in \text{BIP}(X,2\theta_A+\varepsilon).$$
\end{proof}

\subsection{Proof of \refT{Th Pb Cauchy}}

\begin{proof}
Let $C \in \CC$. Then, for all $t \in [0,T]$, setting
$$v(t) = e^{C t} u(t),$$
it is clear that to solve problem \eqref{P} is equivalent to solve the following problem
\begin{equation}\label{P2}
\left\{\begin{array}{l}
\dis v'(t) + (- \A_i + C I)v(t) = g(t), \quad t \in(0,T] \\ \ecart
v(0) = u_0,
\end{array}\right.
\end{equation}
where $g(t) = e^{-C t}f(t) \in L^p(0,T;X)$. Then, setting
$$C = \left\{\begin{array}{ll}
\dis \frac{k^2}{4}, & \text{if } i=1,2, \\ \ecart
\dis \frac{k^2}{4} + r', & \text{if } i=3,4,
\end{array}\right.$$
we deduce from \refT{Th BIP}, for $\varepsilon \in (0,\pi/2 - 2\theta_A)$, that 
$$-\A_i + C I\in \text{BIP}\,(X,2\theta_A + \varepsilon), \quad \text{with }2\theta_A+\varepsilon \in (0,\pi/2).$$
It follows that $\A_i - CI$, is the infinitesimal generator of a strongly continuous analytic semigroup $\left(e^{t(\A_i + CI)}\right)_{t \geqslant 0}$. Moreover, from \eqref{0 in rho(Ai)}, we have $0 \in \rho(-\A_i + CI)$.

Since $g \in L^p(0,T;X)$, $p\in (1,+\infty)$, from \cite{de simon}, Lemma 2.1, p. 208, there exists a unique solution of problem \eqref{P2} given by
\begin{equation}\label{f rep}
v(t) = e^{t(\A_i - CI)} u_0 + \int_0^t e^{(t-s)(\A_i - CI)} g(s)~ds.
\end{equation}
If $v$, given by \eqref{f rep}, is a classical solution of problem \eqref{P2}, then
$$v \in W^{1,p}(0,T;X) \cap L^p\left(0,T;D(\A_i)\right).$$
It follows from \refL{Lem trace} that 
$$v(0) = u_0 \in \left(D(\A_i),X\right)_{\frac{1}{p},p}.$$
Conversely, if $u_0 \in \left(D(\A_i),X\right)_{\frac{1}{p},p}$, since the unique solution $v$ is given by \eqref{f rep}, it remains to show that $v$ is a classical solution. From \refL{Lem triebel}, we deduce that
$$t\longmapsto e^{t(\A_i-CI)}u_0 \in W^{1,p}(0,T;X) \cap L^p\left(0,T;D(\A_i)\right).$$
Finally, from \cite{dore-venni}, Theorem 3.2, p. 196, we obtain that
$$t\longmapsto \int_0^t e^{(t-s)(\A_i-CI)} g(s)~ds \in W^{1,p}(0,T;X) \cap L^p\left(0,T;D(\A_i)\right).$$
\end{proof}

\section{Application}\label{Sect Application}

Let $T>0$ and $\Omega = (a,b)\times \omega$, where $\omega$ is a bounded open set of $\RR^{n-1}$, $n\geqslant2$, with $C^2$-boundary. We focus ourselves on an equation of Cahn-Hilliard type given by
\begin{equation}\label{EQ edp}
\frac{\partial u}{\partial t} (t,x,y) + \Delta^2 u(t,x,y) - k \Delta u(t,x,y) = F(t,u(t,x,y)), \quad t \in [0,T],~x \in (a,b), ~y \in \omega,
\end{equation}
where $k \in \RR$ and
$$F(t,u(t,x,y)) = \frac{k' u(t,x,y)}{k'' + \|u(t,x,y)\|_{L^p((0,T)\times\Omega)}}, \qquad \text{with } k',k''>0\text{ and } p \in (1,+\infty),$$
is used to model, for instance of lactate or oxygen exchanges in glial cells or also the metabolites concentrations in the circadian rhythm in the brain. For more details, we refer to \cite{Miranville 1} or \cite{Miranville 2}.

This equation is supplemented by the initial condition
\begin{equation}
u(0,x,y) = u_0(x,y), \quad (x,y) \in \Omega,
\end{equation}
the homogeneous boundary condition
\begin{equation}\label{BC edp}
u(t,x,\zeta) = \Delta u(t,x,\zeta) = 0, \quad t \in (0,T],~(x,\zeta) \in (a,b) \times \partial\omega,
\end{equation}
and one of the following homogeneous boundary conditions:
\begin{equation}\label{BC edp 1}\tag{BC$_{\text{pde}}$ 1}
\left\{\begin{array}{rclrclc}
u(t,a,y)   & = & 0, & u(t,b,y)   & = & 0, & t \in (0,T],~y \in \omega, \medskip \\
\dis \Delta u(t,a,y) & = & 0, & \dis \Delta u(t,b,y) & = & 0, & t \in (0,T],~y \in \omega,
\end{array}\right.
\end{equation}
\begin{equation}\label{BC edp 2}\tag{BC$_{\text{pde}}$ 2}
\left\{\begin{array}{rclrclc}
\dis\partial_x u(t,a,y) & = & 0, & \dis\partial_x u(t,b,y) & = & 0, & t \in (0,T],~y \in \omega, \medskip \\ 
\dis \Delta u(t,a,y) &= & 0, & \dis \Delta u(t,b,y) &= & 0, & t \in (0,T],~y \in \omega,
\end{array}\right.
\end{equation}
\begin{equation}\label{BC edp 3}\tag{BC$_{\text{pde}}$ 3}
\left\{\begin{array}{rclrclc}
u(t,a,y) & = & 0, & u(t,b,y) & = & 0, &t \in (0,T],~ y \in \omega, \medskip \\
\dis\partial_x u(t,a,y) & = & 0, &\dis \partial_x u(t,b,y) & = & 0, &t \in (0,T],~ y \in \omega,
\end{array}\right.
\end{equation}
or
\begin{equation}\label{BC edp 4}\tag{BC$_{\text{pde}}$ 4}
\left\{\begin{array}{rclrclc}
\dis\partial_x u(t,a,y) & = & 0, & \dis\partial_x u(t,b,y) & = & 0, & t \in (0,T],~y \in \omega, \medskip \\
\dis\partial^2_{x} u(t,a,y) & = & 0, & \dis\partial^2_{x} u(t,b,y) & = & 0, & t \in (0,T],~ y \in \omega.
\end{array}\right.
\end{equation}

We set
\begin{equation*}
\left\{\begin{array}{l}
D(A_0) := W^{2,p}(\omega) \cap W^{1,p}_0(\omega) \\ \ecart
\forall \psi \in D(A_0), \quad A_0 \psi = \Delta_y \psi.
\end{array}\right.
\end{equation*}
Since we have
$$\Delta v(x,y) = \frac{\partial^2 v}{\partial x^2}(x,y) + \Delta_y v,(x,y),$$
then, using operator $A_0$, it follows that 
$$\Delta^2 v(x,y) - k \Delta v(x,y), \quad x \in (a,b), ~y \in \omega,$$
can be written as
$$v^{(4)}(x) + (2A_0 - kI) v''(x) + (A^2_0 - k A_0 ) v(x),$$
where $v(x)=v(x,\cdot)$.

Now, for $i=1,2,3,4$, we deduce that the following linear operators corresponds to the abstract formulation of the spatial operator in \eqref{Pb Cauchy non lineaire}:
$$\left\{\begin{array}{ccl}
D(\A_{0,i}) & = & \{u \in W^{4,p}(a,b;L^p(\omega))\cap L^p(a,b;D(A_0^2))\text{ and } u'' \in L^p(a,b;D(A_0)) : \text{(BCi)}_0\} \\ \ecart
\left[\A_{0,i} u\right](x) & = & - u^{(4)}(x) - (2A_0 - kI) u''(x) - (A^2_0 - k A_0) u(x), \quad u \in D(\A_{0,i}),~x\in (a,b).
\end{array}\right.$$
It follows that, for instance, problem \eqref{EQ edp}-\eqref{BC edp}-\eqref{BC edp 1} writes
\begin{equation}\label{Pb A1}
\left\{\begin{array}{l}
u'(t) - \A_{0,1} u(t) = F(t,u(t)), \quad t\in [0,T] \medskip \\ 
u(0) = u_0,
\end{array}\right.
\end{equation}
where $u(t) = u(t,\cdot) = v(x)$.
 
Then, we have the following result.
\begin{Th}
Let $k \in\RR$ such that $k > -C_\omega$, where $C_\omega >0$ is the Poincar\'e constant in $\omega$. Then, there exists a unique classical global solution of problem \eqref{Pb A1} if and only if
$$u_0 \in (D(\A_{0,1}),X)_{\frac{1}{p},p}.$$
\end{Th}
\begin{proof}
We have to satisfy assumptions of \refC{Cor sol Pb non linéaire}. Here $X=L^p(0,T;L^p(\Omega))$, $p \in (1, +\infty)$. From \cite{rubio}, Proposition~3, p. 207, $X$ satisfies $(H_1)$. Moreover, since $\Omega$ is bounded, $(H_2)$ and $(H_3)$ are satisfied for every $\theta_A \in ~(0,\pi)$, from \cite{pruss-sohr 2}, Theorem~C, p.~166-167. Since $C_\omega$ is the Poincar\'e constant in $\omega$, we have $(-C_\omega,+\infty) \subset \rho(A_0)$. Moreover, $k > -C_\omega$, thus $(H_4)$ is satisfied. Then, thanks to \cite{Miranville 1}, p. 1825 or \cite{Miranville 2}, section 4.5, p. 112, $F$ satisfy the Lipschitz condition of \refC{Cor sol Pb non linéaire}.

Finally, all the assumptions of \refC{Cor sol Pb non linéaire} are satisfied. It follows that, there exists a unique global classical solution of problem \eqref{Pb A1} if and only if 
$$
u_0 \in \left(D(\A_{0,1}),X\right)_{\frac{1}{p},p}.
$$
\end{proof}
Problems \eqref{EQ edp}-\eqref{BC edp}-\eqref{BC edp 2}, \eqref{EQ edp}-\eqref{BC edp}-\eqref{BC edp 3} and \eqref{EQ edp}-\eqref{BC edp}-\eqref{BC edp 4} can also be treated similarly.

\begin{Rem}
For the reader convenience, we make explicit the following interpolation space
$$\left(D(\A_{0,1}),X\right)_{\frac{1}{p},p} = \left\{\begin{array}{ll}
B^{4(1-\frac{1}{p}),p}(a,b;L^p(\omega)) & \text{if } 4(1-\frac{1}{p}) < \frac{1}{p} \Longleftrightarrow 1 < p < \frac{5}{4} \medskip \\

B^{\frac{1}{p},p}_{0,0}(a,b;L^p(\omega)) & \text{if } 4(1-\frac{1}{p}) = \frac{1}{p} \Longleftrightarrow p = \frac{5}{4} \medskip \\

B_0^{4(1-\frac{1}{p}),p}(a,b;L^p(\omega)) & \text{if } \frac{1}{p} < 4(1-\frac{1}{p}) \leqslant 2 + \frac{1}{p} \Longleftrightarrow \frac{5}{4} < p \leqslant \frac{5}{2} \medskip \\

B_{0,0''}^{4(1-\frac{1}{p}),p}(a,b;L^p(\omega)) & \text{if } 2 + \frac{1}{p} < 4(1-\frac{1}{p}) < + \infty \Longleftrightarrow \frac{5}{2} < p < +\infty,
\end{array} \right.,$$
where the Besov space $B^{4(1-\frac{1}{p}),p}(a,b;L^p(\omega))$ is defined in \cite{grisvard 2}, D\'efinition 5.8 and Proposition 5.9, p.~334. Moreover
$$B_{0,0}^{\frac{1}{p},p}(a,b;L^p(\omega)) = \left\{\psi \in B^{\frac{1}{p},p}(a,b;L^p(\omega)) : \int_a^b \frac{\|\psi(x)\|_{L^p(\omega)}^p}{\inf(x-a,b-x)}~dx \right\},$$
$$B_0^{4(1-\frac{1}{p}),p}(a,b;L^p(\omega)) = \left\{\psi \in B^{4(1-\frac{1}{p}),p}(a,b;L^p(\omega)) : \psi(a) = \psi(b) = 0\right\},$$
and 
$$B_{0,0''}^{4(1-\frac{1}{p}),p}(a,b;L^p(\omega)) = \left\{\psi \in B_0^{4(1-\frac{1}{p}),p}(a,b;L^p(\omega)) : \psi''(a) = \psi''(b) = 0\right\}.$$
\end{Rem}




\end{document}